\documentclass[]{article}
\usepackage[english]{babel}
\usepackage{amsmath, amsfonts, amsthm, xspace, xcolor, fullpage, url}

\usepackage{graphicx}

\newcommand{\re}{\mbox{Re}\,\xspace}
\newcommand{\im}{\mbox{Im}\,\xspace}
\newcommand{\norm}[1]{\left \lVert #1 \right \rVert\xspace}

\newtheorem{theorem}{Theorem}[section]
\newtheorem{lemma}[theorem]{Lemma}
\newtheorem{corollary}[theorem]{Corollary}

\newtheorem{question}{Question}

\newtheorem{prop}[theorem]{Proposition}

\theoremstyle{remark}

\theoremstyle{definition}
\newtheorem{hypothesis}{Hypothesis}
\newtheorem{defn}{Definition}[section]
\newtheorem{assumption}[defn]{Assumption}

\numberwithin{equation}{section}

\title{Pseudospectral approximation of Hopf bifurcation for delay differential equations}
\author{Babette de Wolff\footnote{Institut f{\"u}r Mathematik, Freie Universit{\"a}t Berlin, Arnimallee 3, D - 14195 Berlin}, Francesca Scarabel\footnote{LIAM–Laboratory for Industrial and Applied Mathematics, Department of Mathematics and Statistics, York University, 4700 Keele Street, Toronto, ON M3J 1P3, Canada} \footnote{CDLab–Computational Dynamics Laboratory, Department of Mathematics, Computer Science, and Physics, University of Udine, via delle scienze 206, 33100 Udine, Italy}, Sjoerd Verduyn Lunel\footnote{Department of Mathematics, University of Utrecht, Budapestlaan 6, P.O. Box 80010, 3508 TA Utrecht, The
Netherlands}, Odo Diekmann\footnotemark[4].
}

\date{}

\begin{document}

\maketitle

\begin{abstract}
Pseudospectral approximation reduces DDE (delay differential equations) to ODE (ordinary differential equations). Next one can use ODE tools to perform a numerical bifurcation analysis. By way of an example we show that this yields an efficient and reliable method to qualitatively as well as quantitatively analyse certain DDE. To substantiate the method, we next show that the structure of the approximating ODE is reminiscent of the structure of the generator of translation along solutions of the DDE. Concentrating on the Hopf bifurcation, we then exploit this similarity to reveal the connection between DDE and ODE bifurcation coefficients and to prove the convergence of the latter to the former when the dimension approaches infinity.  \\

\end{abstract}

\maketitle



\section{Introduction}
Numerical bifurcation analysis \cite{bifboek, kuznetsov} is nowadays a powerful method for analysing dynamical systems that arise in applications. For ordinary differential equations (ODE) trustworthy  tools, such as Auto \cite{auto} and MatCont \cite{matcont}, exist (here `trustworthy' indicates that they are tested and maintained, i.e., adapted when the software or hardware environment in which they are embedded changes). For delay differential equations (DDE) there are trustworthy tools too, e.g., DDE-BIFTOOL \cite{biftool1, biftool2} and KNUT \cite{knut}, but these can handle only specific classes of DDE, such as equations with point delays, and it seems fair to say that both maintenance and testing is somewhat vulnerable, because it relies on the efforts of just a few individuals, if not just one. So if we manage to systematically approximate infinite dimensional dynamical systems corresponding to DDE by finite dimensional systems corresponding to ODE, we may lose some precision in the numerical bifurcation analysis, but we would be able to handle a much larger class of equations.

In \cite{nonlinearps} pseudospectral approximation is advocated as a promising approach to achieve exactly this. The aim of the present paper is to make a next step by verifying that the generic Hopf bifurcation in DDE is faithfully captured by Hopf bifurcations in the approximating ODE systems. Our theoretical results concern the limit when the dimension of the approximating system goes to infinity. In practice we of course at best verify that a bifurcation diagram remains essentially unchanged when the dimension is increased by a finite amount (for example doubled). The theoretical results generate confidence that the bifurcation diagram of the approximating ODE captures the DDE dynamics if it is robust under increase of the dimension.

In the following we take a famous example from mathematical biology, namely the `Nicholson's blowflies' equation, as a testing ground to illustrate some features of the approach. 
However, we remark that the methodology presented here (pseudospectral approximation combined with software for bifurcation analysis of ODE) can be applied in a much more general setting:
it is indeed a promising procedure to study differential equations with distributed, state-dependent, and even infinite delays \cite{nonlinearps,infinite,statedep}, as well as nonlinear renewal equations \cite{renewalbif} and first order partial differential equations \cite{pde}. The advantage of considering Nicholson's blowflies equation in this context is due to the fact that explicit comparisons are possible, both with analytically computed quantities and with alternative numerical approximations, as will become clear later on.

\section*{Acknowledgments}
We thank two anonymous referees and Sebastiaan Janssens for very helpful comments that led to substantial improvement of the manuscript. 

\section{A motivating example: `Nicholson's blowflies' equation} \label{sec: blowfly}

\begin{figure}[t]
\centering
\includegraphics[scale=1]{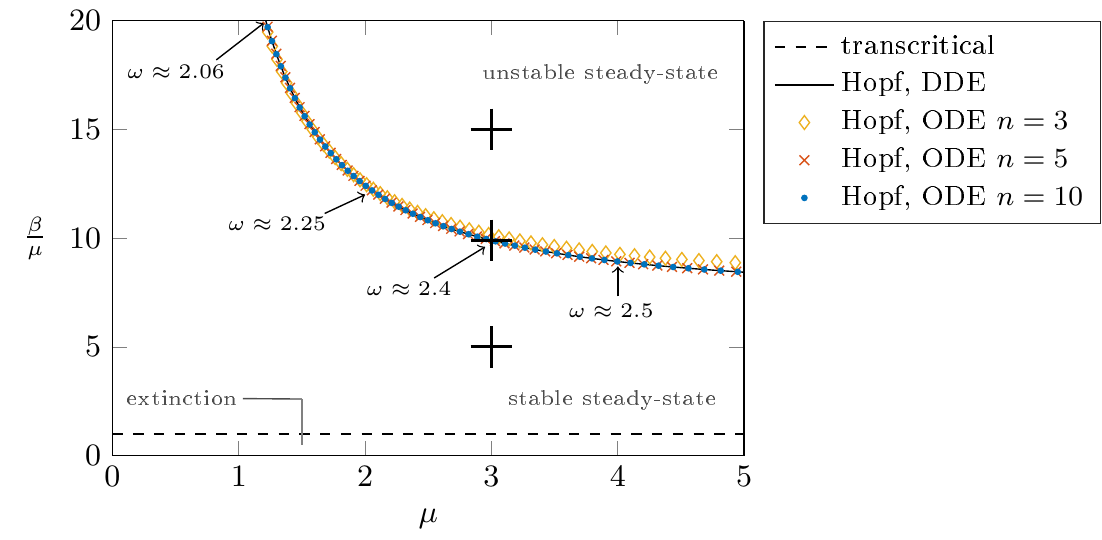}
\caption{Stability diagram of \eqref{eq: blowfly unscaled} and its pseudospectral approximation for $\tau = 1$ and $h(x) = e^{-x}$. The horizontal black dashed line indicates the transcritical bifurcation in \eqref{eq: blowfly unscaled} and its pseudospectral approximation. The Hopf bifurcation curves are computed analytically, both for the DDE (black solid) and the pseudospectral approximation (colors), see Appendix. 
{\color{black}The different values of $\omega$ indicated along the Hopf bifurcation curve specify the position of the critical root of the characteristic equation on the positive imaginary axis.}
The black crosses refer to parameter values used in Figure \ref{fig:eigenvalues}. 
}
\label{fig:hopf-analytic}
\end{figure}

In the paper \cite{blowflies}, Gurney, Blythe and Nisbet showed that Nicholson's classic laboratory blowfly data are in good quantitative agreement with various characteristics of solutions of the DDE
\begin{align} \label{eq: blowfly unscaled}
N'(t) = - \mu N(t) + \beta N(t - \tau) h(N(t - \tau)), \quad t \geq 0. 
\end{align}

Here $N$ corresponds to the size of the population of adults, where newborns become adult after a maturation delay $\tau$. The parameter $\mu \geq 0$ refers to the per capita death rate and $\beta \geq 0$ to the maximum per capita egg production rate. The graph of the recruitment function $N \mapsto N h(N)$ is assumed to be humped. This form reflects scramble competition for the experimentally controlled limited amount of protein resource: female adults need a certain quantity of protein in order to be able to produce eggs.

So \eqref{eq: blowfly unscaled} has a very respectable background in population biology. Here we want to demonstrate that the pseudospectral methodology enables a quick and efficient numerical bifurcation analysis of \eqref{eq: blowfly unscaled} with relatively little effort. In addition we shall pay attention to the accuracy of the approximation. Equation \eqref{eq: blowfly unscaled} is rather well suited to do so, as several features (in particular the stability boundary in a two-parameter space, see Figure \ref{fig:hopf-analytic}) can be derived analytically.

Using the pseudospectral technique, equation \eqref{eq: blowfly unscaled} is approximated by a system of $n+1$ ODE for the variables $y_0,\dots,y_n$, where the first equation reads
\begin{align} \label{eq: blowfly ps part 1}
y_0' = - \mu y_0 + \beta y_n h (y_n),
\end{align}
and captures the rule for extension \eqref{eq: blowfly unscaled}, with $y_0(t)$ and $y_n(t)$ approximating $N(t)$ and $N(t-\tau)$, respectively.
The remaining $n$ equations are needed to describe translation along the solution, and are in fact independent of the specific delay equation.
We refer to Section \ref{sec: ps approximation} for the details of the pseudospectral approximation.

Under the assumption that $h$ is decreasing and vanishing at infinity, with $h(0)=1$, for every $\beta > \mu$ there exists a positive equilibrium of both equation \eqref{eq: blowfly unscaled} and the corresponding approximating system. {\color{black}Moreover, for both equations the stability boundary (in a two-parameter plane) of the positive equilibrium can be computed analytically. We shall do so in Appendix A.}

We find that for $\beta < \mu$ the trivial equilibrium is asymptotically stable and the population goes extinct. For $\beta  = \mu$ the trivial and non-trivial equilibrium exchange stability in a transcritical bifurcation. If we then follow a one-parameter path in the $(\mu, \beta/\mu)$-plane that crosses the Hopf bifurcation curve (see Figure \ref{fig:hopf-analytic}) transversally, the positive equilibrium of \eqref{eq: blowfly unscaled} loses its stability in a Hopf bifurcation. Figure \ref{fig:hopf-analytic} gives the stability diagram for \eqref{eq: blowfly unscaled} and its pseudospectral approximation, for various values of the discretisation parameter $n$. 

One of the main advantages of the pseudospectral approximation is that the resulting system can be analysed with software for the numerical bifurcation analysis of ODE. Throughout the following sections, we will illustrate the obtained results by comparing analytical computations for \eqref{eq: blowfly unscaled} with numerical bifurcation results of the approximating ODE. In Section \ref{sec: outlook} we will explore the dynamics beyond the Hopf bifurcation curve and show that, using numerical approximations, one can transcend a pen-and-paper analysis and investigate more complex objects like periodic solutions and their bifurcations. 

In the following sections we will study the convergence of the approximations in the limit $n\to \infty$. In this perspective, Figure \ref{fig:hopf-analytic} and later figures lift up our spirits by showing that, in practice, the approximation of the stability curves and associated quantities is extremely good already for low values of $n$. 

\section{The Hopf bifurcation theorem: a quick refresher} \label{sec: hopf refresher}

In this section, we recall the Hopf bifurcation theorem for general ODE and for \emph{scalar} DDE. {\color{black} For proofs of (equivalent formulations of) the results, as well as additional references, see \cite[Chapter X, Theorems 2.1, 2.7, 3.1 and 3.9]{sunstarbook} and \cite{Lani-Wayda13}.}

Consider the ODE
\begin{align} \label{eq: ode hopf}
x'(t) = A(\alpha) x(t) + f(x(t), \alpha), \qquad t \geq 0
\end{align}
with $\alpha \in \mathbb{R}$, $A(\alpha): \mathbb{R}^d \to \mathbb{R}^d$ linear and $f: \mathbb{R}^d \times \mathbb{R} \to \mathbb{R}^d$ for some $d \in \mathbb{N}$. We summarise the relevant requirements on $A$ and $f$ in a hypothesis.
\begin{hypothesis} \hfill \label{hyp: condities functies ode}
\begin{enumerate}
\item $f: \mathbb{R}^d \times \mathbb{R} \to \mathbb{R}^d$  and $\alpha \mapsto A(\alpha)$ are $C^k$ smooth for some $k \geq 3$;
\item $f(0, \alpha) = 0$ and $D_1 f(0, \alpha) =0$ for all $ \alpha \in \mathbb{R}$.
\end{enumerate}
\end{hypothesis}
{\color{black}
Under this hypothesis, system \eqref{eq: ode hopf} has an equilibrium $x=  0$ for all $\alpha \in \mathbb{R}$, but in the results presented below only a small neighbourhood of a specific value $\alpha_0$ matters. The linearisation of \eqref{eq: ode hopf} at this equilibrium is given by
\[ \dot{x}(t) = A(\alpha) x(t). \]
For two vectors $v, w \in \mathbb{C}^d$, we define 
\[ v \cdot w = \sum_{i = 1}^n v_i w_i.  \]
Note that this differs from the inner product between $v$ and $w$, which is $v \cdot \overline{w}$ (or $\overline{v} \cdot w$) in the present notation. 
}

\begin{theorem}[Hopf bifurcation theorem for ODE] \label{thm: hopf bif ode}
Consider system \eqref{eq: ode hopf} and assume that Hypothesis \ref{hyp: condities functies ode} is satisfied. If there exist $\alpha_0 \in \mathbb{R}$ and {\color{black}$\omega_0> 0$} such that 
\begin{enumerate}
\item $i \omega_0$ is a simple eigenvalue of $A(\alpha_0)$;
\item the branch of eigenvalues of $A(\alpha)$ through $i\omega_0$ at $\alpha = \alpha_0$ intersects the imaginary axis transversally, i.e., the real part of the derivative of the eigenvalues along the branch is non-zero. 
If we denote by $p, q \in \mathbb{C}^d \backslash \{0 \}$ vectors such that $A(\alpha_0)p = i \omega_0 p, \ A(\alpha_0)^T q = i \omega_0 q$ and $q \cdot p = 1$, then this condition amounts to
\begin{align} \label{eq: crossing eigenvectors}
\re \left(q \cdot   A'(\alpha_0) p\right) \neq 0;
\end{align}
\item {\color{black} $k i \omega_0$ is not an eigenvalue of $A(\alpha_0)$ for $k = 0, 2, 3, \ldots$}
\end{enumerate}
then a Hopf bifurcation occurs for $\alpha = \alpha_0$. This means that there exist $C^{k-1}$ functions $\epsilon \mapsto \alpha^\ast(\epsilon), \ \epsilon \mapsto \omega^\ast(\epsilon)$ taking values in $\mathbb{R}$ and $\epsilon \mapsto x^\ast(\epsilon) \in {\color{black}C_b(\mathbb{R}, \mathbb{R}^d)}$, all defined for $\epsilon$ sufficiently small, such that for $\alpha = \alpha^\ast(\epsilon)$, $x^\ast(\epsilon)$ is a periodic solution of \eqref{eq: ode hopf} with period $2 \pi/\omega^\ast(\epsilon)$. Moreover, $\alpha^\ast$ and $\omega^\ast$ are even functions, $\alpha^\ast(0) = \alpha_0, \ \omega^\ast(0) = \omega_0$ and if $x$ is a small periodic solution of \eqref{eq: ode hopf} for $\alpha$ close to $\alpha_0$ and minimal period close to $2 \pi/\omega_0$, then $x(t) = x^\ast(\epsilon)(t + \theta^\ast)$ and $\alpha = \alpha^\ast(\epsilon)$ for some $\epsilon$ and some $\theta^\ast \in [0, 2 \pi/\omega^\ast(\epsilon))$. 

Moreover, $\alpha^\ast$ has the expansion $\alpha^\ast(\epsilon) =\alpha_{0} + a_{20}  \epsilon^2 + o(\epsilon^2)$, with $a_{20}$ given by
\begin{align*}
a_{20} = - \frac{\re c}{\re \left(q \cdot A'(\alpha_0) p\right)}
\end{align*}
with
\begin{align*}
c = &\frac{1}{2} q \cdot D_1^3 f(0, \alpha_0) (p, p, \overline{p}) + q \cdot D_1^2 f(0, \alpha_0) \bigl(-A(\alpha_0)^{-1} D_1^2 f(0, \alpha_0) (p, \overline{p}), p\bigr) \\
&\qquad + \frac{1}{2} q \cdot D_1^2 f(0, \alpha_0) \bigl( (2 i \omega_0 - A(\alpha_0))^{-1} D_1^2 f(0, \alpha_0) (p, p), \overline{p} \bigr).
\end{align*}
\end{theorem}

\noindent {\color{black}For a proof that condition \eqref{eq: crossing eigenvectors} is equivalent to a transversal crossing of the eigenvalues at the bifurcation point, see \cite[Appendix XIII, Lemma 1.15]{sunstarbook}.}

We refer to the coefficient $a_{20}$ as the \emph{direction coefficient}; {\color{black} the quantity $\frac{1}{\omega_0} \re c$ is usually referred to as the \emph{first Lyapunov coefficient} (as it is the sign that matters, it is tempting to also refer to Re c as the Lyapunov coefficient; below we shall allow ourselves such sloppiness).} In the expression for the direction coefficient, the denominator captures whether {\color{black}dimension of the unstable subspace of the steady state} increases or decreases as we vary the parameter across the bifurcation point. At the bifurcation point, the steady state is not hyperbolic; provided the Lyapunov coefficient is non-zero, it determines whether the steady state is stable or unstable at the bifurcation point \cite{kuznetsov}. 

\medskip

Next we consider the \emph{scalar} DDE
\begin{align}
x'(t) = L(\alpha) x_t + g(x_t, \alpha), \quad t \geq 0  \label{eq: dde parameter}
\end{align}
with state space $X = C\left([-1, 0], \mathbb{R}\right)$, $\alpha \in \mathbb{R}$ a parameter, $L(\alpha) : X \to \mathbb{R}$ a bounded linear operator and $g: X \times \mathbb{R} \to \mathbb{R}$. Without loss of generality, we have taken the maximal delay to be 1. We summarise the relevant requirements on $L$ and $g$ in a hypothesis.
\begin{hypothesis} \label{hyp: condities functies dde} \hfill
\begin{enumerate}
\item $g: X \times \mathbb{R} \to \mathbb{R}$ and $\alpha \mapsto L(\alpha)$ are $C^k$ smooth for some $k \geq 3$;
\item $g(0, \alpha) = 0$ and $D_1 g(0, \alpha)= 0$ for all $\alpha \in \mathbb{R}$.
\end{enumerate}
\end{hypothesis}
Under this hypothesis, system \eqref{eq: dde parameter} has an equilibrium $x= 0$ for all $\alpha \in \mathbb{R}$. The linearisation of \eqref{eq: dde parameter} has a solution $t \mapsto e^{\lambda t}$ if and only if $\lambda$ is a root of
the characteristic equation
\begin{align} \label{eq: ce dde parameter}
\Delta_0(\lambda, \alpha) = 0\qquad\mbox{with}\quad \Delta_0(\lambda, \alpha) := \lambda - L(\alpha)\varepsilon_\lambda,
\end{align}
where $\varepsilon_\lambda \in X$ denotes the exponential function 
\begin{equation}\label{eq:exp-fct}
\varepsilon_\lambda(\theta) = e^{\lambda \theta}, \ \theta \in [-1, 0]. 
\end{equation}
The roots of the characteristic equation \eqref{eq: ce dde parameter} correspond to the eigenvalues of the generator of the linearised semiflow of \eqref{eq: dde parameter}, cf. \cite[Section IV.3]{sunstarbook}.

DDE like \eqref{eq: dde parameter} can have only a finite number of characteristic roots on the imaginary axis, resulting in the existence of a finite dimensional center manifold.
On this finite dimensional center manifold, which is by construction invariant under the flow, the DDE reduces to an ODE. This allows one to `lift' the Hopf bifurcation theorem from ODE to DDE. This is done in detail in \cite[Chapter X]{sunstarbook}, in this section we just state the main result. 

\begin{theorem}[Hopf bifurcation theorem for \emph{scalar} DDE]  \label{thm: hopf bifurcation dde}
Consider equation \eqref{eq: dde parameter} and suppose that Hypothesis \ref{hyp: condities functies dde} is satisfied. If there exist $\alpha_0 \in \mathbb{R}$ and {\color{black}$\omega_0 > 0$} such that
\begin{enumerate}
\item $i \omega_0$ is a simple root of $\Delta_0(\lambda, \alpha_0) = 0$;
\item The branch of roots of $\Delta_0(\lambda,\alpha) = 0$ through $i\omega_0$ at $\alpha = \alpha_0$ intersects the imaginary axis transversally,  i.e. the real part of the derivative of the roots along the branch is non-zero. This condition amounts to
\begin{align*}
\re \left(D_1 \Delta_0(i \omega_0, \alpha_0)^{-1}D_2 \Delta_0(i \omega_0, \alpha_0)\right) \neq 0;
\end{align*}
\item {\color{black}$k i \omega_0$ is not a root of $\Delta_0(\lambda, \alpha_0) = 0$ for $k = 0,2, 3, \ldots$}
\end{enumerate}
then a Hopf bifurcation occurs for $\alpha = \alpha_0$. This means that there exist $C^{k-1}$-functions $\epsilon \mapsto \alpha^\ast(\epsilon)$, $\epsilon \mapsto \omega^\ast(\epsilon)$ taking values in $\mathbb{R}$ and $\epsilon \mapsto x^\ast(\epsilon) \in {\color{black}C_b\left(\mathbb{R}, \mathbb{R}\right)}$, all defined for $\epsilon$ sufficiently small, such that for $\alpha = \alpha^\ast(\epsilon), \ x^\ast(\epsilon)$ is a periodic solution of \eqref{eq: dde parameter} with period $2 \pi/\omega^\ast(\epsilon)$. Moreover, $\alpha^\ast$ and $\omega^\ast$ are even functions, $\alpha^\ast(0) = \alpha_0, \ \omega^\ast(0) = \omega_0$ and if $x$ is a small periodic solution of \eqref{eq: dde parameter} for $\alpha$ close to $\alpha_0$ and minimal period close to $2 \pi/\omega_0$, then $x(t) = x^\ast(\epsilon)(t+\theta^\ast)$ and $\alpha = \alpha^\ast(\epsilon)$ for some $\epsilon$ and some $\theta^\ast \in [0, 2 \pi/\omega^\ast(\epsilon))$. 

Moreover, $\alpha^\ast$ has the expansion 
$\alpha^\ast(\epsilon) =\alpha_0 + a_{20} \epsilon^2 + o(\epsilon^2)$, with $a_{20}$ given by
\begin{align*}
a_{20} = \frac{\re c_0}{\re \left(D_1 \Delta_0(i \omega_0, \alpha_0)^{-1} D_2 \Delta_0(i \omega_0, \alpha_0)\right)}
\end{align*}
where
\begin{equation} \label{eq: c dde}
\begin{aligned}
c_0 = & (D_1 \Delta_0(i \omega_0, \alpha_0))^{-1}\frac{1}{2}D_1^3 g(0, \alpha_0)(\phi, \phi, \overline{\phi}) \\ & \qquad + (D_1 \Delta_0(i \omega_0, \alpha_0))^{-1} D_1^2 g(0, \alpha_0) \bigl(\varepsilon_0 \Delta_0(0, \alpha_0)^{-1} D_1^2 g(0, \alpha_0) (\phi, \overline{\phi}), \phi\bigr) \\
&\qquad + (D_1 \Delta_0(i \omega_0, \alpha_0))^{-1} \frac{1}{2} D_1^2 g(0, \alpha_0) \bigl( \varepsilon_{2 i \omega_0} \Delta_0(2 i \omega_0, \alpha_0)^{-1} D_1^2 g(0, \alpha_0) (\phi, \phi), \overline{\phi}\bigr) 
\end{aligned}
\end{equation}
with $\phi := \varepsilon_{i \omega_0}$. 
\end{theorem}

\section{Pseudospectral approximation} \label{sec: ps approximation}

In order to approximate the infinite dimensional dynamical system corresponding to the DDE \eqref{eq: dde parameter} by a finite dimensional ODE, we first approximate elements of the state space
\begin{align*}
X = C \left([-1, 0], \mathbb{R}\right)
\end{align*}
by polynomials interpolating their values in a chosen set of mesh points. 

{\color{black}Given $n \in \mathbb{N}$ and given a mesh $-1 \leq \theta_n < \ldots < \theta_0 = 0$,} the corresponding \emph{Lagrange polynomials} $\ell_j : [-1,0] \to \mathbb{R}$ are defined by
\begin{align}\label{eq: Lagrange polynomials}
\ell_j(\theta) = \prod_{\substack{0 \leq m \leq n \\ m \neq j}} \frac{\theta - \theta_m}{\theta_j - \theta_m},\quad-1 \le \theta \le 0,\qquad j = 0,1,\ldots,n. 
\end{align}
The properties 
\begin{equation}\label{eq: sum lagrange polynomials}
\sum_{j = 0}^n \ell_j(\theta) \equiv 1\quad\mbox{and}\quad
\ell_j(\theta_i) = \delta_{ij} = \begin{cases}
1 \quad \mbox{if } i = j \\
0 \quad \mbox{if } i \neq j
\end{cases}
\end{equation}
make the Lagrange polynomials suitable building blocks for interpolation{\color{black}, especially since Lagrange interpolation can be implemented in a stable and efficient way by using barycentric interpolation \cite{berrut}.} \\

A DDE is a rule for extending a known history. It defines a dynamical system on the state space of history functions by shifting along the extended function, i.e., by updating the history. This involves that we distinguish the time variable $t$ from the bookkeeping variable $\theta$, needed to describe the history. In particular, we approximate
\begin{align} \label{eq: hist-approx}
x(t+\theta) \sim \sum_{j = 0}^n  \ell_j(\theta) y_j(t),\qquad -1 \le \theta \le 0. 
\end{align}
For the left hand side of \eqref{eq: hist-approx}, the derivative with respect to $t$ equals the derivative with respect to $\theta$. The idea of \emph{collocation} is to require that this is also true for the right hand side of \eqref{eq: hist-approx} at the mesh points $\theta_k, k = 1, \ldots, n$. This condition leads to the following system of differential equations 
\begin{align} \label{eq: collocation}
y_k'(t) = \sum_{j = 0}^n \ell_j'(\theta_k) y_j(t), \quad k = 1, \ldots n. 
\end{align}
By defining 
\begin{align} \label{eq: differentiation matrix}
D: \mathbb{R}^n \to \mathbb{R}^n, \qquad D_{ij} = \ell_j'(\theta_i), \quad i, j = 1, \ldots, n
\end{align}
and taking into account that \eqref{eq: sum lagrange polynomials} implies that
\begin{align*}
\ell_0'(\theta) = - \sum_{j = 1}^n \ell_j'(\theta)
\end{align*}
we can rewrite \eqref{eq: collocation}, using the notation $\mathbf{1}  = (1, \ldots, 1)^T \in \mathbb{R}^n$, as
\begin{align} \label{eq: collocation with D}
y' = D y - y_0 D \mathbf{1}, 
\end{align}
where $y$ is the $n$-vector with components $y_k, k = 1, \ldots, n$. Note that \eqref{eq: collocation with D} is universal in the sense that it does \emph{not} depend on the specific DDE under consideration. 

The differential equation \eqref{eq: collocation with D} approximately captures the translation aspect of the dynamics. The equation for $y_0$ (corresponding to the value of {\color{black}$x_t$} in $\theta_0=0$) captures the specific rule for extension specified by the DDE. Define $P : \mathbb{R}^{n} \to X$ and $P_0 : \mathbb{R} \times \mathbb{R}^n \to X$ as, respectively,
\begin{subequations}
\begin{align} \label{eq: polynomial interpolation}
(Py)(\theta) &= \sum_{j = 1}^n \ell_j(\theta) y_j,\\
\label{eq: polynomial interpolation full}
\bigl(P_0(y_0,y)\bigr)(\theta) &= y_0\ell_0(\theta) + (Py)(\theta),
\end{align}
\end{subequations}
where $\ell_j$, $j=0,1,\ldots,n$, are defined by \eqref{eq: Lagrange polynomials}. We add to  \eqref{eq: collocation with D}
the differential equation
\begin{align}\label{eq: collocation with y0}
y_0' = LP_0(y_0,y) + g\bigl(P_0(y_0,y) \bigr)
\end{align}
to mimic the specific scalar DDE 
\begin{equation} \label{eq: dde}
x'(t) =Lx_t + g(x_t)
\end{equation}
with $L: X \to \mathbb{R}$ bounded linear and $g: X \to \mathbb{R}$. 

So we approximate the infinite dimensional dynamical system corresponding to \eqref{eq: dde} with the finite dimensional dynamical system generated by the ODE \eqref{eq: collocation with D} $\&$ \eqref{eq: collocation with y0}. This is summarised in the following definition:

\begin{defn}
\textup{
The \emph{pseudospectral approximation} to the parameterised DDE  (recall \eqref{eq: dde parameter})
\begin{align}\label{eq: dde parameter 2}
x'(t) = L(\alpha)x_t +g(x_t,\alpha)
\end{align}
is given by the parameterised system of ODE
\begin{align} \label{eq: pseudospectral ode parameter 1}
\frac{d}{dt}\begin{pmatrix}y_0 \\ y \end{pmatrix} = A_n(\alpha)\begin{pmatrix}y_0 \\ y \end{pmatrix} + g(P_0(y_0,y),\alpha) \begin{pmatrix}
1 \\0 
\end{pmatrix} , \quad t \geq 0,
\end{align}
where $A_n(\alpha) : \mathbb{R} \times \mathbb{R}^{n} \to \mathbb{R} \times \mathbb{R}^{n}$ is given by
\begin{align}\label{eq def An}
A_n(\alpha)  = \begin{pmatrix}
L(\alpha) \ell_0 & L(\alpha) P\\
- D \mathbf{1} & D 
\end{pmatrix}.
\end{align}
Here $y_0 \in \mathbb{R}, \ y \in \mathbb{R}^n$, {\color{black}$P$ is defined in \eqref{eq: polynomial interpolation}, $P_0$} is defined in \eqref{eq: polynomial interpolation full}, the matrix $D$ {\color{black}is defined} in \eqref{eq: differentiation matrix}, and the dimension $n$ is a parameter that we have suppressed in the  notation and in the terminology.}
\end{defn}

{\color{black}In the definition above, there is no restriction on the nodes. The theoretical results that we shall present below are, however, based on the following assumption:
\begin{assumption}  \label{assumption}  
If we consider the \emph{reduced} mesh $\{\theta_1,\dots, \theta_n\}$ and the corresponding Lagrange polynomials 
\begin{equation} \label{eq: reduced lagrange polynomials}
\tilde{\ell}_j(\theta) = \prod_{\substack{1 \leq m \leq n \\ m \neq j}} \frac{\theta - \theta_m}{\theta_j - \theta_m},\quad-1 \le \theta \le 0,\qquad j = 1,\ldots,n,
\end{equation}
then the associated \emph{Lebesgue constant} 
\[ \tilde{\Lambda}_n := \max_{\theta \in [-1, 0]} \sum_{j = 1}^n \left| \tilde{\ell}_j(\theta) \right| \] 
satisfies $\lim_{n \to \infty} \frac{\tilde{\Lambda}_n}{n} = 0$.
\end{assumption}
}

{\color{black}
Assumption \ref{assumption} is satisfied by the nodes
\begin{subequations}
\begin{align}
\theta_0 &= 0  \label{eq: node zero}\\
\theta_k &= \frac{1}{2} \left( \cos \left(\frac{2 k- 1}{2n} \pi\right) - 1 \right), \qquad k = 1, \ldots, n, \label{eq: chebyshev zeros}
\end{align}
\end{subequations}
which are the Chebyshev zeros \eqref{eq: chebyshev zeros} with an added node at $\theta = 0$ \cite[Chapter 1.4.6]{MastroianniBook}. 

The numerical computations in this paper are made using the \emph{Chebyshev extremal nodes} 
\begin{align} \label{eq: chebyshev nodes}
\theta_j = \frac{1}{2}\left(\cos\left(\frac{j \pi}{n}\right)-1 \right), \quad 0 \leq j \leq n.
\end{align}
We choose to work with the Chebyshev extremal nodes (rather than with \eqref{eq: node zero}--\eqref{eq: chebyshev zeros}) since for Chebyshev extremal nodes the matrix $D$ in \eqref{eq: differentiation matrix} can be numerically computed in a reliable and efficient way \cite{trefethen}. 
However, if we consider the reduced mesh $\{\theta_1,\dots, \theta_n\}$, then the corresponding Lebesgue constant $\tilde{\Lambda}_n$ is only known to behave like $O(n)$ \cite[Chapter 4.2]{MastroianniBook}. So based on this estimate, the nodes \eqref{eq: chebyshev nodes} do not satisfy Assumption \ref{assumption}. Yet in practice we observe the fast convergence expected from meshes of nodes satisfying Assumption \ref{assumption}, see also \cite{ps-sunstar}. So a remaining challenge is to find an analytical argument that also covers the nodes \eqref{eq: chebyshev nodes}. 

}

\medskip

We remark that if $\overline{x}$ is a steady state of \eqref{eq: dde parameter 2}, then \eqref{eq: pseudospectral ode parameter 1} has a steady state $y_0 = \overline{x}, \ y = \overline{x} \textbf{1}$. {\color{black}Conversely, if $(\overline{y}_0, \overline{y})$ is a steady state of \eqref{eq: pseudospectral ode parameter 1}, then \eqref{eq: collocation} implies that 
\[
P_0(\overline{y}_0, y)'(\theta) = \ell_0'(\theta) \overline{y}_0 + \sum_{j = 1}^n \ell_j'(\theta) y_j
\]
is zero at $\theta = \theta_1, \ldots, \theta_n$. So $P_0(\overline{y}_0, y)'$ is a polynomial of degree $n-1$ with $n$ zeros, which implies that $P_0(\overline{y}_0, y)' \equiv 0$ and $P_0(\overline{y}_0,y)$ is the constant function taking the value $\overline{y}_0$. Therefore $\overline{y}_0$ is a steady state of \eqref{eq: dde parameter 2}.} So, steady states of \eqref{eq: dde parameter 2} and \eqref{eq: pseudospectral ode parameter 1} are in one-to-one correspondence. 

\smallskip

Note that in the pseudospectral approximation \eqref{eq: pseudospectral ode parameter 1}, the nonlinear terms only appear in the equation for $y_0$ and hence the range of the nonlinear perturbation is contained in a one-dimensional subspace. The formula
\begin{equation*}
y(t) = -\int_{-\infty}^t y_0(\tau) e^{(t-\tau)D}  D \mathbf{1}\,d\tau
\end{equation*}
expresses $y$ explicitly in terms of $y_0$ when we consider $y_0$ as given on $(-\infty,t]$. If we substitute this into the differential equation for $y_0$, we obtain a DDE with infinite delay \cite{infinitedelay}. Note that periodic $y_0$ yields periodic (with the same period) $y$. The remark about steady states amounts to: constant $y_0$ yield $y(t) = y_0 \mathbf{1}$.

\subsection*{Characteristic equation}
If $g(0,\alpha) = 0$ and $D_1g(0,\alpha) = 0$, then the linearisation of \eqref{eq: dde  parameter 2} around zero, i.e.,
\begin{align}  \label{eq: linear dde}
x'(t) = L(\alpha)x_t, \quad t \geq 0
\end{align}
has, as mentioned before, a nonzero solution of the form $x(t) = e^{\lambda t}$ if and only if $\lambda$ is a root of the characteristic equation \eqref{eq: ce dde parameter}.

The linearisation of the pseudospectral approximation \eqref{eq: pseudospectral ode parameter 1} of \eqref{eq: dde  parameter 2} around zero has a nontrivial solution of the form $e^{\lambda t} (\zeta_0,\zeta)$ if and only if $\lambda$ is an eigenvalue of \eqref{eq def An} with eigenvector $(\zeta_0, \zeta)$, i.e., if and only if
\begin{subequations}
\begin{align}
\lambda \zeta_0 &= L(\alpha) (\zeta_0 \ell_0 + P\zeta) \label{eq: eigenvalue 1} \\
\lambda \zeta & = D \zeta - \zeta_0 D \mathbf{1} \label{eq: eigenvalue 2}
\end{align}
\end{subequations}
has a nontrivial solution $(\zeta_0, \zeta) \in \mathbb{C}^{n+1}$. {\color{black} We prove in Lemma \ref{lem: convergence collocation} that for $\lambda $ in a given right compact subset of $\mathbb{C}$, $D - \lambda I$ is invertible for $n$ large enough.} Equation \eqref{eq: eigenvalue 2} then implies that 
\begin{equation} \label{eq: def psi}
\zeta = \zeta_0 (D- \lambda I)^{-1} D \mathbf{1}
\end{equation}
and inserting this into \eqref{eq: eigenvalue 1} we obtain that
\begin{align} \label{eq: ce ode 1}
\bigl[\lambda - L(\alpha)\left(\ell_0 + P(D-\lambda I)^{-1} D \mathbf{1}\right)\bigr]\zeta_0 = 0.
\end{align}
This shows that eigenvalues of $A_n(\alpha)$ as defined in \eqref{eq def An} correspond to roots of the characteristic equation 
\begin{align} \label{eq: ce ode}
\Delta_n(\lambda,\alpha) = 0\quad \mbox{with} \quad\Delta_n(\lambda,\alpha) := \lambda - L(\alpha) \left(\ell_0 + P(D-\lambda I)^{-1} D \mathbf{1}\right).
\end{align}
Here the subscript $n$ in the definition of $\Delta_n(\lambda,\alpha)$ specifies the dimension of the approximation. If $\lambda$ is a root of \eqref{eq: ce ode}, then a corresponding eigenvector of $A_n(\alpha)$ is given by 
\begin{equation} \label{eq def p An}
(p_\ast, \tilde{p}) = (1, (D-\lambda I)^{-1} D \textbf{1}). 
\end{equation}

The correspondence between eigenvalues of $A_n(\alpha)$ and roots of $\Delta_n(\lambda, \alpha) = 0$ is analogous to the correspondence between eigenvalues of the generator of translation along solutions of the linearised DDE \eqref{eq: linear dde} and the roots of the characteristic equation $\Delta_0(\lambda, \alpha) = 0$. 

\subsection*{Hopf bifurcation for the pseudospectral approximation}

In order to relate Hopf bifurcation for the DDE \eqref{eq: dde parameter 2} to Hopf bifurcation for the pseudospectral approximation \eqref{eq: pseudospectral ode parameter 1}, we first reformulate Theorem \ref{thm: hopf bif ode} for ODE of the special form \eqref{eq: pseudospectral ode parameter 1}.

The resolvent of $A_n(\alpha) : \mathbb{C} \times \mathbb{C}^{n} \to \mathbb{C} \times \mathbb{C}^{n}$ defined by the complexification of \eqref{eq def An} can be computed explicitly. From $(\lambda I - A_n(\alpha))^{-1}(\zeta_0,\zeta) = (\eta_0,\eta)$ it follows that 
\begin{subequations}
\begin{align}
\zeta_0 &= \lambda \eta_0 - L(\alpha)\ell_0 \eta_0 - L(\alpha)P\eta \label{eq: resolvent ps 1} \\
\zeta & = \lambda \eta +  D \mathbf{1}\eta_0 - D\eta  \label{eq: resolvent ps 2}
\end{align}
\end{subequations}
Since $D - \lambda I$ is invertible for $n$ large enough, we can solve for $\eta$ in terms of $\zeta$ and $\eta_0$ from \eqref{eq: resolvent ps 2}. Substitution of the result in \eqref{eq: resolvent ps 1} then yields
\begin{equation} \label{eq: resolvent ps} 
(\lambda I - A_n(\alpha))^{-1} \begin{pmatrix} \zeta_0 \\ \zeta \end{pmatrix} = \Delta_n(\lambda,\alpha)^{-1} \bigl(\zeta_0 + L(\alpha)P(\lambda I - D)^{-1}\zeta\bigr)\begin{pmatrix}
1 \\ (D-\lambda I)^{-1} D \mathbf{1}
\end{pmatrix} + \begin{pmatrix}
0 \\
(\lambda I - D)^{-1} \zeta
\end{pmatrix}.
\end{equation}

If $\Delta_n(\lambda,\alpha) = 0$ and $D_1\Delta_n(\lambda,\alpha) \not= 0$, the residue of the right hand side of \eqref{eq: resolvent ps} in $\lambda$ defines a projection operator
\begin{equation} \label{eq: projection ps} 
Q_n \begin{pmatrix} \zeta_0 \\ \zeta \end{pmatrix} = D_1\Delta_n(\lambda,\alpha)^{-1} \bigl(\zeta_0 + L(\alpha)P(\lambda I - D)^{-1}\zeta\bigr)\begin{pmatrix}
1 \\ (D-\lambda I)^{-1} D \mathbf{1}
\end{pmatrix} 
\end{equation}
which is of the form
\[Q_n\begin{pmatrix} \zeta_0 \\ \zeta \end{pmatrix} = \bigl(q_\ast \cdot \zeta_0 + \tilde{q} \cdot \zeta \bigr)\begin{pmatrix}
1 \\ (D-\lambda I)^{-1} D \mathbf{1}
\end{pmatrix} 
\]
with $(q_\ast, \tilde{q})$ the adjoint eigenvector to the eigenvalue $\lambda$ of $A_n(\alpha)$, normalised such that
\begin{equation*}
 (q_\ast, \ \tilde{q}) \cdot \begin{pmatrix}
p_\ast \\ \tilde{p}
\end{pmatrix} = 1. 
\end{equation*}
Since $L(\alpha)P \tilde y = \sum_{j=1}^n L(\alpha)\ell_j \tilde y_j$ we find that
\begin{equation} \label{eq def q An} 
q_\ast = \frac{1}{D_1\Delta_n(\lambda,\alpha)} , \qquad \tilde{q} = \frac{1}{D_1\Delta_n(\lambda,\alpha)} (\lambda I - D^T)^{-1}\begin{pmatrix}
L(\alpha)\ell_1 \\ \vdots \\ L(\alpha)\ell_n \end{pmatrix}.
\end{equation}
We can also compute the adjoint eigenvector from \eqref{eq def An}, giving the same result. 

\medskip

Recall the condition
\begin{align*}
\re \left(q \cdot  A'(\alpha_0) p\right) \neq 0,
\end{align*}
in Theorem \ref{thm: hopf bif ode}. From the definition of $A_n(\alpha)$ in \eqref{eq def An} we obtain
\begin{align}\label{eq def An dif}
  A_n' (\alpha)  = \begin{pmatrix}
D_\alpha L(\alpha) \ell_0 & D_\alpha L(\alpha) P\\
0 & 0
\end{pmatrix}.
\end{align}
So using the definitions for the right eigenvector $(p_\ast, \tilde{p})$ in \eqref{eq def p An} and the left eigenvector $(q_\ast, \tilde{q})$ in \eqref{eq def q An} for $\lambda = i \omega$, it follows that
\[(q_\ast, \tilde{q}) \cdot  A_n'(\alpha) \begin{pmatrix}
p_\ast \\ \tilde{p}
\end{pmatrix} = -D_1 \Delta_n(i\omega,\alpha)^{-1}D_2 \Delta_n(i\omega,\alpha).\]
Finally observe from \eqref{eq: pseudospectral ode parameter 1} that the nonlinearity only acts in the first component of the equation.  Therefore the formula for $c$ in Theorem \ref{thm: hopf bif ode} becomes in the present setting
\begin{align*}
c= 
&D_1 \Delta_n(i \omega, \alpha)^{-1}  \frac{1}{2} D_1^3 g(0, \alpha) \bigl(P_0 p, P_0 p, P_0\overline{p}\bigr) \\\ & \qquad +  D_1 \Delta_n(i \omega, \alpha)^{-1}  D_1^2 g(0, \alpha) \Bigl(-P_0 \left( A_n(\alpha)^{-1} \begin{pmatrix}
1 \\ 0\end{pmatrix} \right)  D_1^2 g(0, \alpha) \bigl(P_0p, P_0\overline{p}\bigr), P_0 p\Bigr) \\
&\qquad + D_1 \Delta_n( 2i \omega, \alpha)^{-1}  \frac{1}{2} D_1^2 g(0, \alpha) \Bigl( P_0 \left((2 i \omega - A_n(\alpha))^{-1} \begin{pmatrix}
1 \\0 \end{pmatrix} \right) D_1^2 g(0, \alpha) \bigl(P_0p, P_0p\bigr), P_0\overline{p}\Bigr)
\end{align*}
with $p = (1, (D- i \omega)^{-1} D \textbf{1})$. 
From \eqref{eq: resolvent ps} it follows that
\begin{equation*}
(\lambda I - A_n(\alpha))^{-1} \begin{pmatrix} 1 \\ 0 \end{pmatrix} = \Delta_n(\lambda,\alpha)^{-1}\begin{pmatrix}
1 \\ (D-\lambda I)^{-1} D \mathbf{1}
\end{pmatrix}.
\end{equation*}
We are now ready to apply Theorem \ref{thm: hopf bif ode} to the pseudospectral approximation \eqref{eq: pseudospectral ode parameter 1}.

\begin{theorem}[Hopf bifurcation in pseudospectral ODE] \label{thm: hopf ps}

Consider the system \eqref{eq: pseudospectral ode parameter 1} and suppose that Hypothesis \ref{hyp: condities functies dde} is satisfied. If there exist $\alpha_n \in \mathbb{R}$ and {\color{black}$\omega_n > 0$} such that 
\begin{enumerate}
\item $i \omega_n$ is a simple root of $\Delta_n(\lambda, \alpha_n) = 0$;
\item the branch of roots of $\Delta_n(\lambda, \alpha) = 0$ through $i\omega_n$ at $\alpha = \alpha_n$ intersects the imaginary axis transversally, i.e., the real part of the derivative of the roots along the branch is non-zero. This condition amounts to
\begin{align*}
\re \bigl(D_1 \Delta_n(i\omega_n,\alpha_n)^{-1}D_2 \Delta_n(i\omega_n,\alpha_n) \bigr) \neq 0,
\end{align*}
\item {\color{black}$k i \omega_n$ is not a root of $\Delta_n(\lambda, \alpha_n)=  0$ for $k = 0, 2, 3, \ldots$}
\end{enumerate}
then a Hopf bifurcation occurs for $\alpha = \alpha_n$. 

Moreover, $\alpha^\ast$ as in Theorem \ref{thm: hopf bif ode} has the expansion $\alpha^\ast(\epsilon) = \alpha_n + a_{2n} \epsilon^2 + o(\epsilon^2)$, with $a_{2n}$ given by
\begin{align*}
a_{2n} = \frac{\re c_n}{\re \bigl(D_1 \Delta_n(i\omega_n,\alpha_n)^{-1} D_2 \Delta_n(i\omega_n,\alpha_n) \bigr) }
\end{align*}
with
\begin{equation} \label{eq: c ps}
\begin{aligned}
c_n = 
&D_1 \Delta_n(i \omega_n, \alpha_n)^{-1}  \frac{1}{2} D_1^3 g(0, \alpha_n) \bigl(P_0 p, P_0 p, P_0\overline{p}\bigr)  \\ & \qquad + D_1 \Delta_n(i \omega_n, \alpha_n)^{-1} D_1^2 g(0, \alpha_n)\Bigl( \Delta_n(0,\alpha_n)^{-1} P_0 \begin{pmatrix}
1 \\ \mathbf{1}  
\end{pmatrix}  D_1^2 g(0, \alpha_n) \bigl(P_0 p, P_0\overline{p}\bigr), P_0p\Bigr) \\
&\qquad + D_1 \Delta_n(i \omega_n, \alpha_n)^{-1} \frac{1}{2} D_1^2 g(0, \alpha_n) \Bigl(\Delta_n(2 i\omega_n,\alpha_n)^{-1} P_0\begin{pmatrix}
1 \\ (D-2 i\omega_n I)^{-1} D \mathbf{1} \end{pmatrix} D_1^2 g(0, \alpha_n)\bigl (P_0 p, P_0 p\bigr), P_0\overline{p}\Bigr).
\end{aligned}
\end{equation}
and $p = (1, (D-i \omega_n)^{-1} D \textbf{1})$ the right eigenvector to $A_n( \alpha_n)$ with eigenvalue $i \omega_n$. 
\end{theorem}

\medskip

In the following sections we investigate the issue of convergence.

\section{Approximation of spectral data of linear problems}

Comparing the characteristic equations \eqref{eq: ce dde parameter} and \eqref{eq: ce ode}, we see that the following variant of a result from \cite[Lemma 3.2]{bmv}, \cite[Proposition 5.1]{bmvboek} is relevant; we include its proof for completeness. 

\begin{lemma} \label{lem: convergence collocation}
{\color{black}Let $U \subseteq \mathbb{C}$ be a compact subset. Then there exist a positive integer $N = N(U)$ and a constant $C > 0$ such that for $n \geq N$ and $\lambda \in U$, $D-\lambda I$ is invertible and 
\begin{align} \label{eq: error estimate}
\norm{\ell_0  + P ( D-\lambda I)^{-1} D \textbf{1} - \varepsilon_\lambda}  \leq \frac{1}{\sqrt{n}} \left( \frac{C}{n}\right)^n 
\end{align}
with $\varepsilon_\lambda$ defined as in \eqref{eq:exp-fct}.
}
\end{lemma}

\begin{proof}
{\color{black}
Fix $\lambda \in U$ and $\zeta_0 \in \mathbb{C}$. We want to solve
\begin{equation} \label{eq: inverteerbaarheid D-lambda}
(D-\lambda I) \zeta = \zeta_0 D \textbf{1}
\end{equation}
for $\zeta \in \mathbb{C}^n$. If $\zeta$ satisfies \eqref{eq: inverteerbaarheid D-lambda}, then $d := \ell_0 \zeta_0 + P \zeta$ is a polynomial of degree $n$ that satisfies
\begin{align*}
d'(\theta_k) &= \ell_0'(\theta_k) \zeta_0 + \sum_{j = 1}^n \ell_j'(\theta_k) \zeta_j \\
& = \left(-\zeta_0 D \textbf{1} \right)_k + \left(D \zeta\right)_k \\
& = \lambda \zeta_k \\
&= \lambda d(\theta_k)
\end{align*}
for $k = 1, \ldots, n$. So $d$ has to satisfy
\begin{align} \label{eq: d}
\begin{cases}
d'(\theta) &= \lambda d(\theta), \quad \theta = \theta_1,\ldots, \theta_n\\d(0) & = \zeta_0.
\end{cases}
\end{align}
Vice versa, if $d$ is a polynomial of degree $n$, then 
\[ d(\theta) = \sum_{j = 0}^n \ell_j(\theta) \zeta_j \]
with $\zeta_j = d(\theta_j)$, $j = 0, \ldots, n$. So if $d$ additionally satisfies \eqref{eq: d}, then 
\[ \ell_0'(\theta_k) \zeta_0 + \sum_{j = 1}^n \ell_j'(\theta_k) \zeta_j = \lambda \zeta_k \]
for $k = 1, \ldots, n$, and $\zeta = (d(\theta_1), \ldots, d(\theta_n))$ is a solution of \eqref{eq: inverteerbaarheid D-lambda}. So finding a solution $\zeta \in \mathbb{C}^n$ of \eqref{eq: inverteerbaarheid D-lambda} is equivalent to finding a polynomial of degree $n$ that satisfies \eqref{eq: d}. 

\medskip

Define the operators
\begin{align*}
L_n: X \to X, &\qquad L_n \phi = \sum_{j = 1}^n \tilde{\ell}_j(.) \phi(\theta_j), \\
K: X \to X, &\qquad (K \phi)(\theta) = \int_0^\theta \phi(s) ds 
\end{align*}
with $\tilde{\ell}_j$ as in \eqref{eq: reduced lagrange polynomials} for $j = 1, \ldots, n$. If $d$ is a polynomial of degree $n$ satisfying \eqref{eq: d}, then 
\begin{equation} \label{eq: d 2}
d' = \lambda L_n d.
\end{equation}
Since $d(\theta) = (K d)'(\theta) +\zeta_0$ for $\theta \in [-1, 0]$, \eqref{eq: d 2} gives
\begin{equation*}
d' = \lambda L_n K d' + \lambda \zeta_0, 
\end{equation*}
where $\zeta_0$ denotes the function taking the constant value $\zeta_0$ and where we have used that $L_n \zeta_0 = \zeta_0$. So if $d$ is a polynomial of degree $n$ satisfying \eqref{eq: d}, then $d$ solves
\begin{subequations}
\begin{align}
d'(\theta) &= \lambda (L_n K d')(\theta) + \lambda \zeta_0, \quad \theta \in [-1, 0], \label{eq: d'} \\
d(0) &= \zeta_0 \label{eq: d' 2}. 
\end{align}
\end{subequations}
Vice versa, if $d$ is a solution of \eqref{eq: d'}--\eqref{eq: d' 2}, then $d'$ is a polynomial of degree $n-1$ and therefore $d$ is a polynomial of degree $n$. Moreover, for $k = 1, \ldots, n$ we find that
\begin{align*}
d'(\theta_k) &= \lambda (Kd')(\theta_k) + \lambda \zeta_0 \\
 & = \lambda (d(\theta_k) - \zeta_0) + \lambda \zeta_0 \\
 & = \lambda d(\theta_k)
\end{align*}
so $d$ satisfies \eqref{eq: d}. We conclude that $d$ is a polynomial of degree $n$ satisfying \eqref{eq: d} if and only if $d$ solves \eqref{eq: d'}--\eqref{eq: d' 2}. 

\medskip

Define $y: = \varepsilon_\lambda \zeta_0$, then $y$ satisfies
\begin{align} \label{eq: y}
\begin{cases}
y'(\theta) = \lambda y(\theta), \quad \theta \in [-1, 0]\\
y(0) = \zeta_0.
\end{cases}
\end{align}
Since $y(\theta) = (Ky')(\theta) +\zeta_0$, $\theta \in [-1, 0]$, \eqref{eq: y} gives
\begin{equation} \label{eq: y'}
y' = \lambda  K y' + \lambda \zeta_0
\end{equation}
where $\zeta_0$ denotes the function taking the constant value $\zeta_0$. 
Now suppose that $d$ satisfies \eqref{eq: d'}--\eqref{eq: d' 2}. Then $e_n := d' - y'$ satisfies 
\begin{equation} \label{eq: e_n}
e_n = \lambda L_n K e_n + \lambda \left(L_n - I\right) K y'.
\end{equation}
Vice versa, if $e_n$ satisfies \eqref{eq: e_n}, then $d' : = e_n + y'$ satisfies \eqref{eq: d'} and hence $d(\theta) := (K d')(\theta) + \zeta_0, \ \theta \in [-1, 0]$ satisfies \eqref{eq: d'}--\eqref{eq: d' 2}. 

\medskip

{\color{black}For $\phi \in X$, $K\phi$ is a Lipschitz function. Since by Assumption \ref{assumption} the Lebesgue constant $\tilde{\Lambda}_n$ associated to the nodes $\{\theta_1,\dots,\theta_n\}$ satisfies $\lim_{n\to\infty} \frac{\tilde{\Lambda}_n}{n} = 0$, it follows from standard interpolation theory that $\lim_{n \to \infty} L_n K = K$ in operator norm, see for example \cite[Sections 4.1--4.2]{interpolation}.}

Since $K$ is Volterra, $(I - \lambda K)$ is invertible for $\lambda \in \mathbb{C}$. Therefore $(I - \lambda L_n K)$ is invertible for $n$ large enough and $\lim_{n \to \infty} ( I - \lambda L_n K)^{-1} = (I - \lambda K)^{-1}$. From here it follows for $n$ large enough, \eqref{eq: e_n} has a unique solution $e_n$: 
\begin{equation} \label{eq: e_n solved}
e_n = ( I - \lambda L_n K)^{-1} \lambda  \left(L_n- I\right) K y'.
\end{equation}
Thus, there is a unique function $d' = e_n' + y'$ satisfying \eqref{eq: d'} and therefore a unique function $d(\theta) := (K d')(\theta)+ \zeta_0, \ \theta \in [-1, 0]$ satisfying \eqref{eq: d'}--\eqref{eq: d' 2}. So there is a unique $\zeta \in \mathbb{C}^n$ satisfying \eqref{eq: inverteerbaarheid D-lambda}. 

For $\zeta_0 = 0$, this implies that the kernel of $D-\lambda I$ is trivial and hence the map $D-\lambda I: \mathbb{C}^n \to \mathbb{C}^n$ is invertible. So we can now also truthfully write $\zeta = \zeta_0(D-\lambda I)^{-1} D \textbf{1}$. 

{\color{black}Standard error estimates for polynomial interpolation (note that $K y'$ is analytic)} give that
$$ \norm{ \left(L_n - I\right) K y'} \leq C_1 \frac{\left| \lambda \right|^n}{n!} \left| \zeta_0 \right|,$$
for some $C_1 > 0$; see for example \cite[Theorem 1.5]{interpolation}. Moreover, since $\lim_{n \to \infty} ( I - \lambda L_n K)^{-1} = (I - \lambda K)^{-1}$, the sequence $(\|(I-\lambda L_n K)^{-1}\|)_{n \in \mathbb{N}}$ is bounded. So \eqref{eq: e_n solved} gives that
\[ \norm{e_n} \leq C_2\frac{\left| \lambda \right|^n}{n!} \left| \zeta_0 \right| \]
for some $C_2 > 0$. Together with Stirling's formula this then yields the error estimate \eqref{eq: error estimate} for $\zeta_0 = 1$.  
}
\end{proof}

\begin{corollary} \label{cor: ce}
Let $\Delta_0(\lambda, \alpha)$ and $\Delta_n(\lambda, \alpha)$ be given by, respectively, \eqref{eq: ce dde parameter} and \eqref{eq: ce ode}. 
{\color{black}Let $U \subseteq \mathbb{C} \times \mathbb{R}$ be a compact subset. Then there exists a $C > 0$ such that
\begin{align*}
\left| \Delta_0(\lambda, \alpha) - \Delta_n(\lambda, \alpha) \right| < \frac{1}{\sqrt{n}} \left(\frac{C}{n}\right)^n
\end{align*}
for $n \in \mathbb{N}$ large enough and $(\lambda, \alpha) \in U$. }
\end{corollary}

\medskip
Next we will exploit the fact that both $\Delta_0$ and $\Delta_n$ are analytic functions in $\lambda$ to prove convergence of the derivatives as well, as $n$ tends to infinity. First an auxiliary lemma.

\begin{lemma} \label{lemma: derivatives}
Let $h_0 : \mathbb{C} \to \mathbb{C}$ and $h_n :  \mathbb{C} \to \mathbb{C}$, $n \in \mathbb{N}$, be analytic functions. Assume 
\[h_0(z) =\lim_{n \to \infty} h_n(z)\qquad\mbox{uniformly for } z \mbox{ in compact subsets of } \mathbb{C}.\]
{\color{black}Fix a compact subset $U \subseteq \mathbb{C}$ and let $V \subseteq \mathbb{C}$ be a compact set such that $U$ is contained in the interior of $V$. Let $(\rho_n)_{n \in \mathbb{N}} = \left(\rho_n(V)\right)_{n \in \mathbb{N}}$ be a sequence such that
\[ \left| h_n(z) - h_0(z) \right| \leq \rho_n \qquad \mbox{for all } n \in \mathbb{N} \mbox{ and } z \in V.\]
Moreover, fix $k \in \{0, 1, 2, \ldots \}$ and denote the $k$-th derivative of $h$ by $h^{(k)}$. Then there exists a constant $C_k > 0$ such that 
\begin{align*}
\left| h_n^{(k)} (z) - h^{(k)}_0(z) \right| \leq C_k \rho_n
\end{align*}
for $n \in \mathbb{N}$ and $z \in U$. }
\end{lemma}

\begin{proof}
By the Cauchy Integral Formula, we have that
\begin{align*}
h_n(z) = \frac{1}{2\pi i} \int_{\partial V} \frac{h_n(s)}{(s - z)}ds, \qquad h_0(z) = \frac{1}{2 \pi i} \int_{\partial V} \frac{h_0(s)}{(s-z)} ds
\end{align*}
for all $z \in U$. This yields that 
\begin{align*}
h_n^{(k)}(z) = \frac{1}{2 \pi i} k! \int_{\partial V} \frac{h_n(s)}{(s-z)^{k+1}} ds, \qquad h_0^{(k)}(z) = \frac{1}{2 \pi i} k! \int_{\partial V} \frac{h_0(s)}{(s-z)^{k+1}} ds
\end{align*}
for $k \in \{0, 1, 2, \ldots \}$ and $z \in U$.
Since $U, V$ are compact sets and $U$ is contained in the interior of $V$, we find that there exists a $\delta > 0$ such that $\left| z - s \right| > \delta$ for all $z \in U, s \in \partial V$. Thus, we see that
{\color{black}
\begin{align*}
\left| h_n^{(i)} (z) - h_0^{(i)}(z) \right|  &= \frac{1}{2 \pi} k! \left| \int_{\partial V} \frac{h_n(s) - h_0(s)}{(s-z)^{k+1}} ds \right| \\
& \leq \frac{1}{2 \pi} k! \frac{1}{\delta^{k+1}} \tilde{C} \rho_n
\end{align*}}
for some $\tilde{C} > 0$, which proves the claim.  
\end{proof}

\begin{corollary} \label{cor: convergence derivatives ce}
{\color{black} Let $\Delta_0(\lambda, \alpha)$ and $\Delta_n(\lambda, \alpha)$ be given by, respectively, \eqref{eq: ce dde parameter} and \eqref{eq: ce ode}. Let $U \subseteq \mathbb{C} \times \mathbb{R}$ be a compact subset. Then there exists a $C > 0$ such that
\begin{align*}
\left| D_1 \Delta_0(\lambda, \alpha) - D_1 \Delta_n(\lambda, \alpha) \right| < \frac{1}{\sqrt{n}} \left(\frac{C}{n}\right)^n
\end{align*}
for $n \in \mathbb{N}$ large enough and $(\lambda, \alpha) \in U$.  }
\end{corollary}

\begin{figure}
\centering
\includegraphics[scale=1]{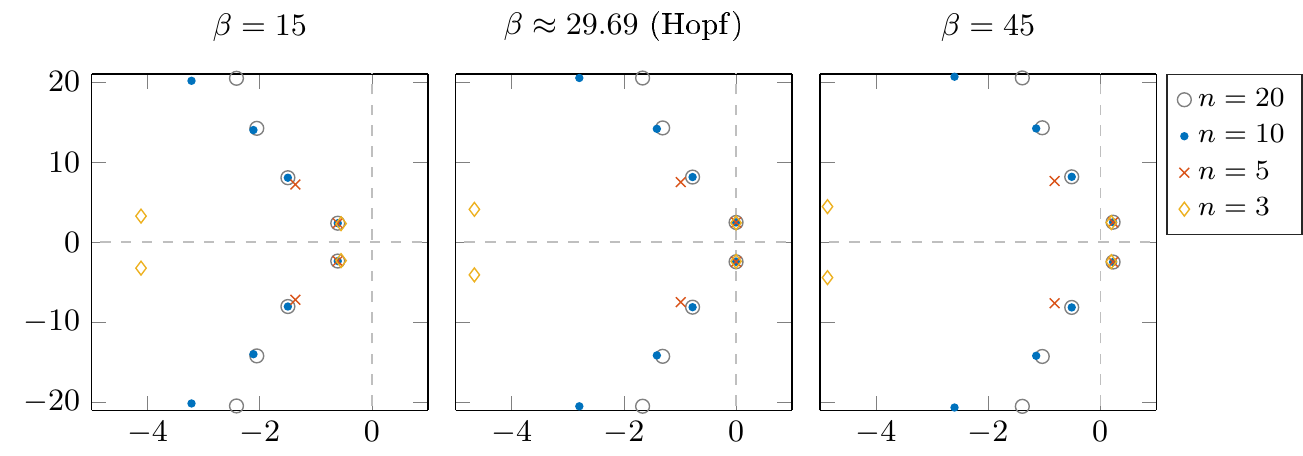}
\caption{{\color{black}Pseudospectral approximation to \eqref{eq: blowfly unscaled}} with $\tau = 1$ and $h(x) = e^{-x}$: {\color{black}roots of the characteristic equation} at the positive equilibrium for $\mu=3$ and different values of $\beta$ as indicated at the top (corresponding to the three black crosses in Figure \ref{fig:hopf-analytic}). The eigenvalues are approximated with MatCont.
}
\label{fig:eigenvalues}
\end{figure}

\section{Hopf bifurcation in the pseudospectral limit}
\label{sec:conv_hopf}

\begin{figure}
\centering
\includegraphics[scale=1]{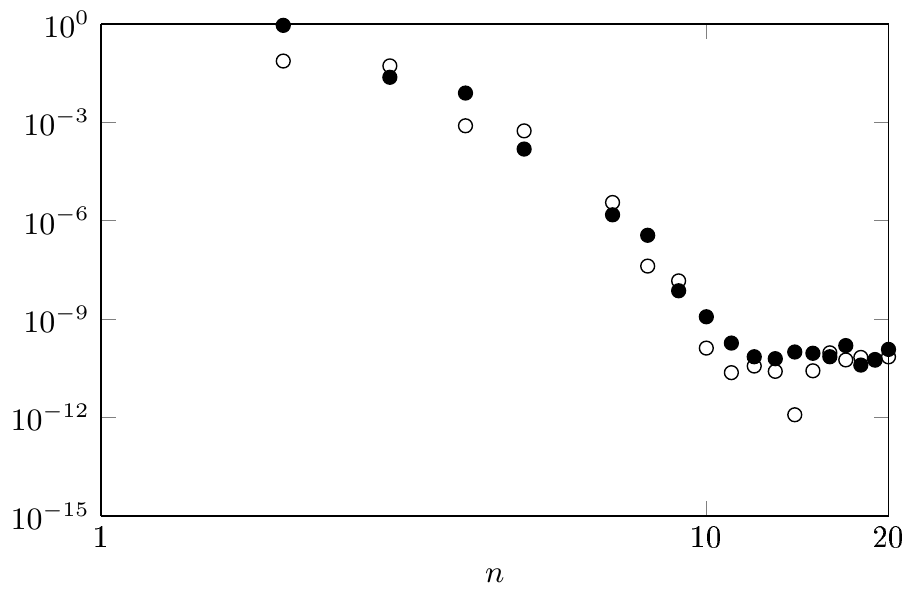}
\caption{Equation \eqref{eq: blowfly unscaled} with $\tau = 1$ and $h(x) = e^{-x}$: log-log plot of the error in the detection of Hopf point (bullets) and in the approximation of the imaginary part of the rightmost {\color{black}roots of the characteristic equation} at Hopf (circles), at $\mu=3$. The errors are calculated by requiring a tolerance of $10^{-9}$ in MatCont computations, and by calculating the absolute value of the difference between the MatCont output and the analytic values. Note the exponential decay until the accuracy $10^{-10}$ is reached.
}
\label{fig:errors}
\end{figure}

In the following, we denote a generic Hopf bifurcation by the triple $(\alpha, \omega, a_2)$, where $\alpha$ is the bifurcation point, $i \omega$ the {\color{black}root of the characteristic equation} on the imaginary axis and $a_2$ the direction coefficient. Here we use the word generic to indicate the three standard conditions ($1.$ simple {\color{black}root of the characteristic equation}; $2.$ transversal crossing; $3.$ non-resonance) and we do \emph{not} require that the direction coefficient is non-zero. To show that the Hopf bifurcation in the pseudospectral approximation is a faithful representation of the Hopf bifurcation in the DDE, we have to answer the following questions:

\begin{question} \label{question 1}
If the DDE has a generic Hopf bifurcation $(\alpha_0, \omega_0, a_{20})$, do the pseudospectral ODE have Hopf bifurcations $(\alpha_n, \omega_n, a_{2n})$ with $\lim_{n \to \infty} (\alpha_n, \omega_n, a_{2n}) = (\alpha_0, \omega_0, a_{20})$?
\end{question}

\begin{question} \label{question 2}
Vice versa, if the pseudospectral ODE have generic Hopf bifurcations $(\alpha_n, \omega_n, a_{2n})$ with 
\[\lim_{n \to \infty} (\alpha_n, \omega_n, a_{2n}) = (\alpha_0, \omega_0, a_{20}),\]
does the DDE have a Hopf bifurcation $ (\alpha_0, \omega_0, a_{20})$?
\end{question}

Answering these questions involves checking the following conditions:

\begin{enumerate}
\item  At the bifurcation point, there is a simple {\color{black}root of the characteristic equation} on the imaginary axis.
\item This {\color{black}root of the characteristic equation} on the imaginary axis crosses the axis transversely if we vary the parameter. 
\item At the bifurcation point, there are no {\color{black}roots of the characteristic equation} in resonance with the root on the imaginary axis. 
\item Convergence of the direction coefficients.
\end{enumerate}
\medskip
We first answer Question \ref{question 1}. To check conditions $1$ and $2$, we use the following lemma, which can be viewed as a version of the Implicit Function Theorem with a (discrete) parameter living in $\mathbb{N}$. It is inspired by \cite[Theorem A.1]{ift}  where the parameter belongs to a general metric space. 

\begin{lemma} \label{lem: discrete IFT}
Let $h_0: \mathbb{R}^d \to \mathbb{R}^d$ and $h_n: \mathbb{R}^d \to \mathbb{R}^d, \ n \in \mathbb{N},$ be $C^1$ functions with 
\begin{equation} \label{eq: ift aanname}
h_0(x) = \lim_{n \to \infty} h_n(x) \quad \mbox{and} \quad Dh_0(x) = \lim_{n \to \infty} Dh_n(x) 
\end{equation}
uniformly for $x$ in compact subsets of $\mathbb{R}^d$. 
Given a compact subset $U \subset \mathbb{R}^d$, let $(\rho_n)_{n \in \mathbb{N}} = (\rho_n(U))_{n \in \mathbb{N}}$ be a sequence such that
\begin{align} \label{eq: convergence rate}
 \norm{h_n(x) - h_0(x)} \leq \rho_n \quad \mbox{for all } n \in \mathbb{N} \mbox{ and } x \in U.
\end{align}
Assume that there exists $x_0 \in \mathbb{R}^d$ such that $h_0(x_0) = 0$ and $Dh_0(x_0)$ is invertible. Then there exists a sequence $(x_n)_{n \in \mathbb{N}}$ such that for $n$ large enough, $h_n(x_n) =0$ and $Dh_n(x_n)$ is invertible. Moreover,  {\color{black} there exists a constant $C > 0$ such that
\begin{align*}
\norm{x_n - x_0} \leq C \rho_n, \quad n \in \mathbb{N}\quad\mbox{ with } \rho_n = \rho_n(U).
\end{align*}}
\end{lemma}

\begin{proof}
Define the functions
\begin{align*}
f_0(x) = x - Dh_0(x_0)^{-1} h_0(x), \quad f_n(x) = x - Dh_0(x_0)^{-1} h_n(x)
\end{align*}
so that zero's of $h_n, h_0$ correspond to fixed points of $f_n, f_0$, respectively. 
Note that $D f_0(x_0) = 0$ and 
\[\lim_{n \to \infty} Df_n(x) = Df_0(x)\qquad\mbox{uniformly for } x \mbox{ in compact subsets}.\]
Therefore we can find a $\rho > 0$ and a $0 < q < 1$ such that $\|Df_n(x)\| < q$ for all $n\in \mathbb{N}, \ x \in B(x_0, \rho)$.  From the Mean Value Theorem we obtain that, for all $n \in \mathbb{N}$, $f_n: B(x_0, \rho) \to \mathbb{R}^d$ is Lipschitz with Lipschitz constant $q$. From the Contraction Mapping Principle, it follows that for all $n \in \mathbb{N}$, $f_n$ has a unique fixed point $x_n$ in $B(x_0, \rho)$. 
Moreover, if we let $U$ be a neighbourhood of $x_0$ and $(\rho_n)_{n \in \mathbb{N}}$ be as in \eqref{eq: convergence rate}, then  
\begin{align*}
\norm{x_n - x_0} & \leq  \norm{f_n(x_n) - f_n(x_0)} + \norm{f_n(x_0) - f_0(x_0)} \\
& < q \norm{x_n - x_0} + \norm{Dh_0(x_0)^{-1}} \rho_n.
\end{align*}
This yields the estimate
\begin{align*}
\norm{x_n - x_0} \leq \frac{\rho_n}{1-q} \| Dh_0(x_0)^{-1} \| . 
\end{align*}
Moreover, since $\lim_{n \to \infty} Dh_n(x_n) = Dh_0(x_0)$ and $Dh(x_0)$ is invertible, $Dh_n(x_n)$ is invertible for $n$ large enough. 
\end{proof}

\begin{prop} \label{cor: approx bif pt}
Consider system \eqref{eq: dde parameter} and suppose that Hypothesis \ref{hyp: condities functies dde} is satisfied. Moreover, suppose that there exist $\alpha_0 \in \mathbb{R}$ and {\color{black}$\omega_0 >0$} such that
\begin{enumerate}
\item $i \omega_0$ is a simple root of $\Delta_0(\lambda, \alpha_0) = 0 $;
\item The branch of roots of $\Delta_0(\lambda,\alpha)  = 0$ through $i\omega_0$ at $\alpha = \alpha_0$ intersects the imaginary axis transversally, i.e.,
\begin{align} \label{eq: crossing ift}
\re \left(D_1 \Delta_0(i \omega_0, \alpha_0)^{-1}D_2 \Delta_0(i \omega_0, \alpha_0)\right) \neq 0. 
\end{align}
\end{enumerate}
Then, for $n$ large enough, there exist $\alpha_n \in \mathbb{R}$, {\color{black}$\omega_n >0$} such that 
\begin{enumerate}
\item $i \omega_n$ is a simple root of $\Delta_n(\lambda, \alpha_n) = 0$; 
\item the branch of roots of $\Delta_n(\lambda, \alpha) = 0$ through $i\omega_n$ at $\alpha = \alpha_n$ intersects the imaginary axis transversally, i.e.
\begin{align} \label{eq: crossing approx}
\re \bigl(D_1 \Delta_n(i\omega_n,\alpha_n)^{-1} D_2 \Delta_n(i\omega_n,\alpha_n) \bigr) \neq 0.
\end{align}
\end{enumerate}
Moreover, there exists a $C > 0$ such that
\begin{align} \label{eq: error estimate prop}
\norm{(\alpha_n, \omega_n) - (\alpha_0, \omega_0)} \leq \frac{1}{\sqrt{n}} \left( \frac{C}{n}\right)^n, \quad \mbox{for } n \in \mathbb{N} \mbox{ large enough}.
\end{align}
\end{prop}
\begin{proof}
Define the functions
$h_n, h_0: \mathbb{R}^2 \to \mathbb{R}^2$ as
\begin{align*}
h_n(\omega, \alpha) = \begin{pmatrix}
\re \Delta_n(i \omega, \alpha) \\
\im \Delta_n (i \omega, \alpha)
\end{pmatrix}, \qquad h_0(\omega, \alpha) = \begin{pmatrix}
\re \Delta_0( i \omega, \alpha) \\
\im \Delta_0(i \omega, \alpha)
\end{pmatrix}.
\end{align*}
Then $h_0(\omega_0, \alpha_0) = 0$ and \eqref{eq: ift aanname} is satisfied by Corollary \ref{cor: ce} and Corollary \ref{cor: convergence derivatives ce}. In order to apply Lemma \ref{lem: discrete IFT}, we only have to check that $Dh_0(\omega_0, \alpha_0)$ is invertible. 

For $\sigma, \omega \in \mathbb{R}$, write
\begin{equation*}
\Delta_0(\sigma + i \omega, \alpha_0) = f_1(\sigma, \omega)+ i f_2 (\sigma, \omega)
\end{equation*}
with $f_1, f_2 \in \mathbb{R}$. 
With this notation $Dh_0(\omega_0, \alpha_0)$ becomes
\begin{equation*}
Dh_0(\omega_0, \alpha_0) = \begin{pmatrix}
D_2 f_1(0, \omega_0) & \re D_2  \Delta_0(i \omega_0, \alpha_0) \\
D_2 f_2(0, \omega_0) & \im D_2  \Delta_0(i \omega_0, \alpha_0)
\end{pmatrix}. 
\end{equation*}
The Cauchy-Riemann equations read
\begin{equation*}
D_2 f_1(\sigma, \omega) = - D_1 f_2 (\sigma, \omega), \qquad D_2 f_2(\sigma, \omega) = D_1 f_1 (\sigma, \omega)
\end{equation*}
and hence
\begin{equation} \label{eq: Dh}
Dh_0(\omega_0, \alpha_0) = \begin{pmatrix}
- D_1 f_2 (0, \omega_0) & \re D_2  \Delta_0(i \omega_0, \alpha_0) \\
D_1 f_1 (0, \omega_0) & \im D_2  \Delta_0(i \omega_0, \alpha_0)
\end{pmatrix}.
\end{equation}
But now note that if we compute $D_1 \Delta_0(i \omega_0, \alpha_0)$, we may as well compute the difference quotient by taking the limit over the real axis, so
\begin{equation*}
 \re D_1 \Delta_0(i \omega_0, \alpha_0) = D_1 f_1 (0, \omega_0), \qquad \im D_1 \Delta_0(i \omega_0, \alpha_0) = D_1 f_2(0, \omega_0)
\end{equation*}
and \eqref{eq: Dh} becomes
\begin{equation*}
Dh_0(\omega_0, \alpha_0) = \begin{pmatrix}
-\im D_1 \Delta_0(i \omega_0, \alpha_0)  & \re D_2 \Delta_0(i \omega_0, \alpha_0) \\
\re D_1 \Delta_0(i \omega_0, \alpha_0) &\im  D_2 \Delta_0(i \omega_0, \alpha_0)
\end{pmatrix}.
\end{equation*}
The invertibility of the matrix $D h_0(\omega_0, \alpha_0)$ is equivalent to the condition \eqref{eq: crossing ift}. So we can apply Lemma \ref{lem: discrete IFT} to find a sequences $(i \omega_n)_{n \in \mathbb{N}}, (\alpha_n)_{n \in \mathbb{N}}$ with $\Delta_n(i \omega_n, \alpha_n) = 0$ and with the error estimate \eqref{eq: error estimate prop}. Moreover, a similar argument as before gives that the invertibility of $Dh_n(\omega_n, \alpha_n)$ is equivalent to the condition \eqref{eq: crossing approx}. 
\end{proof}

For equation \eqref{eq: blowfly unscaled}, the statements of Proposition \ref{cor: approx bif pt} are illustrated in Figure \ref{fig:eigenvalues}--\ref{fig:errors}. In Figure \ref{fig:eigenvalues}, the {\color{black}roots of the characteristic equation} of the pseudospectral approximation of \eqref{eq: blowfly unscaled} are plotted for different values of the parameter $\beta$. Figure \ref{fig:errors} shows the error in the detection of the Hopf point and the imaginary part of the {\color{black}root of the characteristic equation} for the pseudospectral approximation. We see that the desired tolerance level is obtained for relatively low values of the discretisation index ($n \approx 10$).  

\medskip 

Next we look at the non-resonance condition. Suppose that $\Delta_0(i \omega_0, \alpha_0) = 0$ but $\Delta_0(k i \omega_0, \alpha_0) \neq 0$ for all $k = 0, 2, \ldots$. Corollary \ref{cor: ce} gives that for fixed $k$, there exists a $N = N(k)$ such that $\Delta_n(k i\omega_n, \alpha_n) \neq 0$ for $n \geq N(k)$. However, this does not imply that we can choose this $N$ to be uniform in $k$, i.e., that we can find a $N$ such that 
\begin{align} \label{eq: nonresonancy}
\Delta_n(k i\omega_n, \alpha_n) \neq 0 \quad \mbox{for all } n \geq N \mbox{ and all }\ k = 0, 2, 3, \ldots.
\end{align}
So Corollary \ref{cor: ce} does not exclude that for every $n \in \mathbb{N}$ large enough there exists a $k(n)$ such that $\Delta_n(k(n) i \omega_n, \alpha_n) = 0$. This is clearly a non-generic situation, but in order to answer the third condition listed below Question \ref{question 2}, we have to exclude it explicitly. See also Section \ref{sec: outlook}. 

\medskip

Concerning the convergence of the direction coefficient we find: 

\begin{lemma} \label{lem: convergence direction coefficient}
Consider system \eqref{eq: dde parameter} and suppose that the hypotheses of Theorem \ref{thm: hopf bifurcation dde} are satisfied. Let $(\alpha_n, \omega_n)$ be as in Proposition \ref{cor: approx bif pt}. Then $\lim_{n \to \infty} a_{2n}  = a_{20}$. Moreover, if the nonlinearity $g: X \times \mathbb{R} \to X$ is $C^4$, then there exists a $C > 0$ such that
\begin{align*}
\left| a_{2n} - a_{20} \right| \leq \frac{1}{\sqrt{n}} \left( \frac{C}{n} \right)^n \qquad \mbox{for } n \in \mathbb{N} \mbox{ large enough}.  
\end{align*}
\end{lemma}
\begin{proof}
Throughout the proof, we use the symbol $C$ to denote a generic constant whose actual value may differ from line to line. For instance, an upper bound $C_1 C^n$, with $C_1 >  1$ , is replaced by the upperbound $C^n$ , with the second $C$ slightly larger than the first $C$.

{\color{black}We first prove that $\lim_{n \to \infty} c_n = c_0$, with $c_n$ defined as in \eqref{eq: c ps} and $c_0$ defined as in \eqref{eq: c dde}.} Given a compact neighbourhood $U$ of $i \omega_0$, Lemma \ref{lem: convergence collocation} gives a constant $C > 0$ such that
\begin{align} \label{eq: schatting 1}
\left \lVert \varepsilon_\lambda - P_0(1, (D- \lambda I)^{-1} D \textbf{1}) \right \rVert \leq \frac{1}{\sqrt{n}} \left( \frac{C}{n}\right)^n
\end{align}
for all $\lambda \in U$. By Proposition \ref{cor: approx bif pt}, there exists a $C > 0$ such that $\left \lVert (i \omega_n, \alpha_n) - (i \omega_0, \alpha_0) \right \rVert <  \frac{1}{\sqrt{n}} \left( \frac{C}{n}\right)^n$. Since the map $\lambda \mapsto \varepsilon_\lambda(\theta)$ is locally Lipschitz continuous, uniformly for $\theta \in [-1, 0]$, we can find a $C > 0$ such that 
\begin{align} \label{eq: schatting 2}
\left \lVert \varepsilon_{i \omega_0} - \varepsilon_{i \omega_n} \right \rVert \leq \frac{1}{\sqrt{n}} \left( \frac{C}{n}\right)^n
\end{align}
holds. 
Using \eqref{eq: schatting 1} and \eqref{eq: schatting 2} we obtain the estimate
\begin{equation}
\begin{aligned} \label{eq: schatting 3}
\left \lVert \varepsilon_{i \omega_0} - P_0(1, (D - i \omega_n I)^{-1} D \textbf{1}) \right \rVert &\leq \left \lVert \varepsilon_{i \omega_0} - \varepsilon_{i \omega_n} \right \rVert + \left \lVert \varepsilon_{i \omega_n} - P_0(1, (D - i \omega_n I)^{-1} D \textbf{1}) \right \rVert \\
& \leq  \frac{1}{\sqrt{n}} \left( \frac{C}{n}\right)^n. 
\end{aligned}
\end{equation}
We compare the first term of $c_n$ defined in \eqref{eq: c ps} with the first term of $c_0$ defined in \eqref{eq: c dde}. Writing $p = (1, (D- i \omega_n)^{-1} D \textbf{1})$ and $\phi = \varepsilon_{i \omega_0}$, we estimate
\begin{equation}
\begin{aligned} \label{eq: schatting 5}
\left \lVert D_1^3 g(0, \alpha_n)(P_0 p, P_0 p, P_0 \overline{p}) - D_1^3 g(0, \alpha_0) (\phi, \phi, \overline{\phi}) \right \rVert \leq  & \left \lVert D_1^3 g(0, \alpha_n) (\phi, \phi, \overline{\phi}) - D_1^3 g(0, \alpha_0) (\phi, \phi, \overline{\phi}) \right \rVert  \\\ &+ \left \lVert  D_1^3 g(0, \alpha_n) (\phi, \phi, \overline{\phi}) -  D_1^3 g(0, \alpha_n)(P_0 p, P_0 p, P_0 \overline{p}) \right \rVert
\end{aligned}
\end{equation}
Since the map $\alpha \mapsto D_1^3 g(0, \alpha)$ is continuous and $\alpha_n \to \alpha_0$ as $n \to \infty$, we obtain that \[\left \lVert D_1^3 g(0, \alpha_n) (\phi, \phi, \overline{\phi}) - D_1^3 g(0, \alpha_0) (\phi, \phi, \overline{\phi}) \right \rVert \to 0\] as $n \to \infty$. If $g$ is $C^4$, then the map $\alpha \mapsto D_1^3 g(0, \alpha) (\phi, \phi, \overline{\phi})$ is locally Lipschitz and we obtain 
\begin{equation}
\begin{aligned} \label{eq: schatting 4}
\left \lVert D_1^3 g(0, \alpha_n) (\phi, \phi, \overline{\phi}) - D_1^3 g(0, \alpha_0) (\phi, \phi, \overline{\phi}) \right \rVert &\leq C \left| \alpha_n - \alpha_0 \right| \\
& \leq  \frac{1}{\sqrt{n}} \left( \frac{C}{n}\right)^n. 
\end{aligned}
\end{equation}
Since the map $(u, v, w) \mapsto D_1^3 g(0, \alpha_n)(u, v, w)$ is linear in every argument, we can rewrite

\begin{align*}
D_1^3 g(0, \alpha_n) (\phi, \phi, \overline{\phi}) -  D_1^3 g(0, \alpha_n)(P_0 p, P_0 p, P_0 \overline{p})  \\ = D_1^3 g(0, \alpha_n)(\phi - P_0 p, \phi, \overline{\phi}) + D_1^3 g(0, \alpha_n) (P_0 p, \phi - P_0 p, \overline{\phi}) + D_1^3 g(0, \alpha_n) (P_0 p, P_0 p, \overline{\phi} - P_0 \overline{p}).
\end{align*} 
Combining this with \eqref{eq: schatting 3}, we obtain the estimate
\begin{align} \label{eq: schatting 6}
\left \lVert D_1^3 g(0, \alpha_n) (\phi, \phi, \overline{\phi}) -  D_1^3 g(0, \alpha_n)(P_0 p, P_0 p, P_0 \overline{p}) \right \rVert \leq \frac{1}{\sqrt{n}} \left( \frac{C}{n}\right)^n. 
\end{align}
So from  \eqref{eq: schatting 5}, \eqref{eq: schatting 4} and \eqref{eq: schatting 6} we conclude that
\begin{align*}
\lim_{n \to \infty} D_1^3 g(0, \alpha_n)(P_0 p, P_0 p, P_0 \overline{p}) = D_1^3 g(0, \alpha_0) (\phi, \phi, \overline{\phi})
\end{align*}
and if $g$ is $C^4$, then
\begin{align} \label{eq: schatting 7}
\left \lVert D_1^3 g(0, \alpha_n)(P_0 p, P_0 p, P_0 \overline{p}) - D_1^3 g(0, \alpha_0) (\phi, \phi, \overline{\phi}) \right \rVert \leq \frac{1}{\sqrt{n}} \left( \frac{C}{n}\right)^n.  
\end{align}
{\color{black}Now suppose that $(x_n)_{n \in \mathbb{N}} \subseteq \mathbb{C}, \ (y_n)_{n \in \mathbb{N}}  \subseteq \mathbb{C} \backslash \{0 \}$ are sequences with $\lim_{n \to \infty} x_n  = x_0, \ \lim_{n \to \infty} y_n = y_0 \neq 0$.} Then we find for their fraction
\begin{align} \label{eq: breuk}
\left| \frac{x_n}{y_n} - \frac{x_0}{y_0} \right| = \left| \frac{x_n y_0 - x_0 y_n}{y_n y_0} \right|  \leq \left| \frac{(x_n - x_0) y_0}{y_n y_0} \right| + \left| \frac{x_0 (y_n - y_0)}{y_n y_0} \right| \leq C \left(\left| x_n - x_0 \right| + \left| y_n - y_0 \right| \right).
\end{align}
By Corollary \ref{cor: convergence derivatives ce}, there exists a $C > 0$ such that
\begin{align*}
\left| D_1 \Delta_n(i \omega_n, \alpha_n) - D_1 \Delta_0(i \omega_0, \alpha_0) \right| \leq \frac{1}{\sqrt{n}} \left( \frac{C}{n}\right)^n. 
\end{align*}
So if we apply \eqref{eq: breuk} with $x_n = D_1^3 g(0, \alpha_n)(P_0 p, P_0 p, P_0 \overline{p})$ and $y_n = D_1 \Delta_n(i \omega_n, \alpha_n)$, we see that
\begin{align*}
\left \lVert D_1 \Delta_n(i \omega_n, \alpha_n)^{-1} D_1^3 g(0, \alpha_n)(P_0 p, P_0 p, P_0 \overline{p}) - D_1 \Delta_0(i \omega_0, \alpha_0)^{-1}  D_1^3 g(0, \alpha_0) (\phi, \phi, \overline{\phi}) \right \rVert \to 0 \mbox{ as } n \to \infty
\end{align*}
and if $g$ is $C^4$, then 
\begin{align*}
\left \lVert D_1 \Delta_n(i \omega_n, \alpha_n)^{-1} D_1^3 g(0, \alpha_n)(P_0 p, P_0 p, P_0 \overline{p}) - D_1 \Delta_0(i \omega_0, \alpha_0)^{-1}  D_1^3 g(0, \alpha_0) (\phi, \phi, \overline{\phi}) \right \rVert
\leq \frac{1}{\sqrt{n}} \left( \frac{C}{n}\right)^n.
\end{align*}

Applying similar arguments to the second and third term of $c_n$, we find that $\lim_{n \to \infty} c_n = c_0$; if $g$ is $C^4$, we obtain the error estimate
\begin{align*}
\left| c_n - c_0 \right| \leq \frac{1}{\sqrt{n}} \left( \frac{C}{n}\right)^n.
\end{align*}

To analyse the convergence of the direction coefficient $a_{2n}$, we apply \eqref{eq: breuk} with $x_n = \re c_n$ and $y_n = \re \left(D_1 \Delta_n(i \omega_n, \alpha_n)^{-1} D_2 \Delta_n(i \omega_n, \alpha_n)\right)$. We conclude that 
\begin{align*}
\left| a_{2n} - a_{20} \right| \to 0 \quad \mbox{as } n \to \infty
\end{align*}
and if $g$ is $C^4$, then 
\begin{align*}
\left| a_{2n} - a_{20} \right| \leq \frac{1}{\sqrt{n}} \left( \frac{C}{n}\right)^n
\end{align*}
which proves the claim. 
 \end{proof}

Summarising, we find the following answer to Question \ref{question 1}:

\begin{prop} \label{prop 1}
Consider system \eqref{eq: dde parameter} and suppose that the hypotheses of Theorem \ref{thm: hopf bifurcation dde} are satisfied. Moreover, with $\alpha_n, \omega_n$ as in Proposition \ref{cor: approx bif pt}, assume that
\begin{align} \label{eq: non-resonance}
\mbox{for } n \in \mathbb{N} \mbox{ large enough, } \qquad \Delta_n(k i \omega_n, \alpha_n) \neq 0 \qquad \mbox{ for } k = 0, 2, 3, \ldots
\end{align}
Then the hypotheses of Theorem \ref{thm: hopf ps} are satisfied and $\lim_{n \to \infty} a_{2n} = a_{20}$. Moreover, if the nonlinearity $g: X \times \mathbb{R} \to X$ is $C^4$, then there exists a $C > 0$ such that
\begin{align*}
\left| a_{2n} - a_{20} \right| \leq \frac{1}{\sqrt{n}} \left( \frac{C}{n} \right)^n \qquad \mbox{for } n \in \mathbb{N} \mbox{ large enough}. 
\end{align*}
\end{prop}

We now consider Question \ref{question 2}. {\color{black}Suppose that we have sequences $(\alpha_n)_{n \in \mathbb{N}}$, $(\omega_n)_{n \in \mathbb{N}}$
with $\lim_{n \to \infty} \alpha_n = \alpha \in \mathbb{R}, \ \lim_{n \to \infty} \omega_n = \omega_0 \neq 0$. Suppose that $i\omega_n$ is a simple root of $\Delta_n(\lambda, \alpha_n) = 0$ and such that this root crosses the axis transversely if we vary $\alpha$. Then $\Delta_0(i \omega_0, \alpha_0) = 0$ but we have to make additional assumptions to make sure that this root is simple and it crosses the axis transversely if we vary $\alpha$.} Similarly, if for $n \in \mathbb{N}$ large enough, it holds that $\Delta_n(k i \omega_n, \alpha_n) \neq 0$ for $k = 0, 2, 3, \ldots$, we have to make additional assumptions to ensure that $\Delta_0(k i \omega_0, \alpha_0) \neq 0$ for $k = 0, 2, 3\ldots$. 

\begin{prop} \label{prop 2}
Consider system \eqref{eq: dde parameter} and suppose that there exists a $N_0 \in \mathbb{N}$ such that for $n \in\mathbb{N}, \ n \geq N_0$ the hypotheses of Theorem \ref{thm: hopf ps} are satisfied with $\lim_{n \to \infty} \alpha_{n} = \alpha_0, \ {\color{black}\lim_{n \to \infty} \omega_n = \omega_0 \neq 0}$ and $\lim_{n \to \infty} a_{2n} = a_{20}'$. Moreover, suppose that
\begin{enumerate}
\item The sequence $( D_1 \Delta_n(i \omega_n, \alpha_n))_{n \geq N_0}$ is uniformly bounded away from zero;
\item The sequence $\left(\re \left( D_1 \Delta_n( i\omega_n, \alpha_n)^{-1} D_2 \Delta_n(i \omega_n, \alpha_n)\right)\right)_{n \geq N_0}$ is uniformly bounded away from zero;
\item For each $k = 0, 2, 3 \ldots$, the sequence $(\Delta_n(k i \omega_n, \alpha_n))_{n \geq N_0}$ is uniformly bounded away from zero. 
 \end{enumerate} Then the Hypotheses of Theorem \ref{thm: hopf bifurcation dde} are satisfied and the direction coefficient is given by $a_{20}'$, i.e. $a_{20} = a_{20}'$. 
\end{prop}
\begin{proof} Taking the limit in $\Delta_n( i \omega_n, \alpha_n) = 0$ gives that $\Delta_0(i \omega_0, \alpha_0) = 0$. 
The conditions (1), (2) and (3) ensure that $D_1 \Delta_0(i \omega_0, \alpha_0) \neq 0$, $\re \left( (D_1 \Delta(i \omega_0, \alpha_0))^{-1} D_2 \Delta(i \omega_0, \alpha_0) \right) \neq 0$ and $\Delta_0(k i \omega_0, \alpha_0) \neq 0$ for $k = 0, 2, 3, \ldots$. Moreover, as in the proof of Lemma \ref{lem: convergence direction coefficient} 
we find that $\lim_{n \to\infty} a_{2n} = a_{20}$, which implies that $a_{20} = a_{20}'$. 
\end{proof}

\section{Systems}
We formulate the relevant definitions and results for systems of DDE. 

Let $d \in \mathbb{N}$ and consider the system
\begin{align} \label{eq: dde system}
x'(t) = L(\alpha)x_t + g(x_t, \alpha), \quad t \geq 0
\end{align}
with state space $X = C \left([-1, 0], \mathbb{R}^d\right)$, $\alpha \in \mathbb{R}$ a parameter, $L(\alpha): X \to \mathbb{R}^d$ a bounded linear operator and $g: X \times \mathbb{R} \to \mathbb{R}^d$. We summarise the relevant assumptions on $L$ and $g$ in the following hypothesis:
\begin{hypothesis} \hfill \label{hyp: condities functies system}
\begin{enumerate}
\item $g: X \times \mathbb{R} \to \mathbb{R}^d$ and $\alpha \to L(\alpha)$ are $C^k$ smooth for some $k \geq 3$;
\item $g(0, \alpha) = 0$ and $D_1 g(0, \alpha) =0$ for all $\alpha \in \mathbb{R}$.
\end{enumerate}
\end{hypothesis}

\medskip

{\color{black}Under this hypothesis, \eqref{eq: dde system} has an equilibrium $x = 0$ for all $\alpha \in \mathbb{R}$.} The linearisation of \eqref{eq: dde system} has a solution of the form $t \mapsto e^{\lambda t} c, \ c \in \mathbb{C}^d$ if and only if {\color{black}$\lambda$ is a root of the characteristic equation}
\begin{align*}
\det \Delta_0(\lambda, \alpha) = 0
\end{align*}
{\color{black}where the operator $\Delta_0(\lambda, \alpha): \mathbb{C}^d \to \mathbb{C}^d$ is defined as}
\begin{align} \label{eq: ce system}
\Delta_0(\lambda, \alpha) = \lambda I_d - L(\alpha) \varepsilon_\lambda
\end{align}
with $I_d: \mathbb{C}^d \to \mathbb{C}^d$ is the identity operator and $\varepsilon_\lambda$ defined as in \eqref{eq:exp-fct}. In \eqref{eq: ce system}, $L(\alpha) \varepsilon_\lambda$ maps $\mathbb{C}^d $ to $\mathbb{C}^d$ in the following way: given $v \in \mathbb{C}^d$, the function $(\varepsilon_\lambda v)(\theta) = \varepsilon_\lambda(\theta)v$ is an element of $C \left([-1, 0], \mathbb{C}^d \right)$; then $L(\alpha) \left(\varepsilon_\lambda v \right)$ is a vector in $\mathbb{C}^d$. 

If $i \omega_0$ is a simple root of $\det \Delta_0(\lambda, \alpha_0) = 0$, then $\Delta_0(i \omega_0, \alpha_0)$ has a one-dimensional kernel. Moreover, if $p, q \in \mathbb{C}^d \backslash \{0 \}$ are such that $\Delta_0(i \omega_0, \alpha_0) p = 0, \ \Delta_0(i \omega_0, \alpha_0)^T q = 0$, then $q \cdot D_1 \Delta_0(i \omega_0, \alpha_0) p  \neq 0$, see \cite[Exercise IV.3.12]{sunstarbook}. In particular, we can (and will) scale $p, q$ such that $q \cdot D_1 \Delta_0(i \omega_0, \alpha_0) p  = 1$. 

\begin{theorem}[Hopf bifurcation theorem for \emph{systems} of DDE]  \label{thm: hopf bifurcation dde systems}
Consider system \eqref{eq: dde system} and suppose that Hypothesis \ref{hyp: condities functies system} is satisfied. Moreover, suppose that there exist $\alpha_0 \in \mathbb{R}$ and {\color{black}$\omega_0 >0$} such that
\begin{enumerate}
\item $i \omega_0$ is a simple root of $\det \Delta_0(\lambda, \alpha_0) = 0 $;
\item The branch of roots of $\det \Delta_0(\lambda,\alpha)  = 0$ through $i\omega_0$ at $\alpha = \alpha_0$ intersects the imaginary axis transversally, i.e., the real part of the derivative of the roots along the branch is non-zero. If we denote by $p, q \in \mathbb{C}^d \backslash \{0 \}$ the vectors such that $\Delta_0(i \omega_0, \alpha_0)p= 0, \ \Delta_0(i \omega_0, \alpha_0)^T q = 0$ and $q \cdot D_1 \Delta_0(i \omega_0, \alpha_0) p =1$, then this condition amounts to
\begin{align*}
\re \left(q \cdot D_2 \Delta_0(i \omega_0, \alpha_0) p \right) \neq 0;
\end{align*}
\item {\color{black}$k i \omega_0$ is not a root of $\det \Delta_0(\lambda, \alpha_0)$ for $k = 0, 2, 3, \ldots$}
\end{enumerate}
Then a Hopf bifurcation occurs for $\alpha = \alpha_0$. This means that there exist $C^{k-1}$-functions $\epsilon \mapsto \alpha^\ast(\epsilon)$, $\epsilon \mapsto \omega^\ast(\epsilon)$ taking values in $\mathbb{R}$ and {\color{black}$\epsilon \mapsto x^\ast(\epsilon) \in C_b\left(\mathbb{R}, \mathbb{R}^d\right)$}, all defined for $\epsilon$ sufficiently small, such that for $\alpha = \alpha^\ast(\epsilon), \ x^\ast(\epsilon)$ is a periodic solution of \eqref{eq: dde system} with period $2 \pi/\omega^\ast(\epsilon)$. Moreover, $\alpha^\ast, \omega^\ast$ are even functions, $\alpha^\ast(0) = \alpha_0, \ \omega^\ast(0) = \omega_0$ and if $x$ is any small periodic solution of\eqref{eq: dde system} for $\alpha$ close to $\alpha_0$ and minimal period close to $2 \pi/\omega_0$, then $x(t) = x^\ast(\epsilon)(t+\theta^\ast)$ and $\alpha = \alpha^\ast(\epsilon)$ for some $\epsilon$ and some $\theta \in [0, 2 \pi/\omega^\ast(\epsilon))$. 

Moreover, $\alpha^\ast$ has the expansion 
$\alpha^\ast(\epsilon) = \alpha_0 + a_{20} \epsilon^2 + o(\epsilon^2)$, with $a_{20}$ given by
\begin{align*}
a_{20} = \frac{\re c}{\re \left(q \cdot  D_2 \Delta_0(i \omega_0, \alpha_0)p\right)}
\end{align*}
where
\begin{equation} \label{eq: c dde system}
\begin{aligned}
c = & \frac{1}{2} q \cdot D_1^3 g(0, \alpha_0)(\phi, \phi, \overline{\phi}) \\ & \qquad + q \cdot D_1^2 g(0, \alpha_0) \bigl(\varepsilon_0 \Delta_0(0, \alpha_0)^{-1} D_1^2 g(0, \alpha_0) (\phi, \overline{\phi}), \phi\bigr) \\
&\qquad +\frac{1}{2} q\cdot D_1^2 g(0, \alpha_0) \bigl(\varepsilon_{2 i \omega_0} \Delta_0(2 i \omega_0, \alpha_0)^{-1} D_1^2 g(0, \alpha_0) (\phi, \phi), \overline{\phi}\bigr) \bigr)
\end{aligned}
\end{equation}
with $\phi := \varepsilon_{i \omega_0} p$. 
\end{theorem}

\medskip

To write down the pseudospectral approximation to \eqref{eq: dde system}, let for $j = 0, \ldots, n$
\begin{align*}
y_j(t) \in \mathbb{R}^d
\end{align*}
and denote the components of this vector as
\begin{align*}
y_j(t)(k), \qquad k = 1, \ldots d. 
\end{align*}
We define the interpolation operators $P: \mathbb{R}^{nd} \to X, \ P_0: \mathbb{R}^d \times \mathbb{R}^{nd} \to X$ componentwise as
\begin{align*}
\left( Py \right)_k(\theta) &:= \sum_{j = 1}^n \ell_j(\theta) y_j(k) \\
\left(P_0 (y_0, y) \right)_k &:= \ell_0(\theta) y_0(k) + \left( Py \right)_k(\theta)
\end{align*}
{\color{black}where $\ell_j, \ j = 0, 1, \ldots, n$ are defined by \eqref{eq: Lagrange polynomials}.}
We approximate 
\begin{align*}
x_k(t+ \theta) \sim \sum_{j = 0}^n \ell_j(\theta) y_j(t)(k), \qquad k = 1, \ldots d
\end{align*}
and by collocation on the meshpoints $\theta_1, \ldots, \theta_n$ we obtain
\begin{align} \label{eq: collocation systems}
y_i'(t)(k) = \sum_{j = 1}^n D_{ij} y_j(t)(k) - y_0(t)(k) \left[ D \textbf{1}\right]_i, \qquad i = 1, \ldots, n 
\end{align}
with $D$ as in \eqref{eq: differentiation matrix}. 
To approximate the rule for extension, we supplement \eqref{eq: collocation systems} with
\begin{align} \label{eq: extension systems}
y_0'(t) = L(\alpha) P_0(y_0, y) + g(P_0(y_0, y), \alpha). 
\end{align}

Suppressing the index $i$ in the notation we write \eqref{eq: collocation systems} as  
\begin{align} \label{eq: abbreviation}
y'(t)(k) = D y(t)(k)- y_0(t)(k) D \textbf{1} , \qquad k = 1, \ldots d
\end{align}
and next, by suppressing $k$, abbreviate to 
\begin{align*}
y' = Dy - y_0 D \textbf{1}
\end{align*}
where this expression is to be understood $d$-componentwise as in \eqref{eq: abbreviation}. With this notation, the pseudospectral approximation to \eqref{eq: dde system} becomes
\begin{equation} \label{eq: ps approx system}
\begin{aligned}
y_0'(t) &= L(\alpha) P_0(y_0, y) + g(P_0(y_0, y), \alpha), \\
y'(t) &= Dy(t) - y_0(t) D \textbf{1} .
\end{aligned}
\end{equation}

The linearisation of \eqref{eq: ps approx system} around $x = 0$ has a solution of the form $\varepsilon_\lambda (\zeta_0, \zeta)$ if and only if
\begin{subequations}
\begin{align}
\lambda \zeta_0 & = L(\alpha) \ell_0 \zeta_0 + L(\alpha) P \zeta \label{eq: ce ps system 1} \\
\lambda \zeta & = D \zeta- \zeta_0 D \textbf{1}   \label{eq: ce ps system 2}
\end{align}
\end{subequations}
with $\zeta_j \in \mathbb{C}^d$ for $j = 0, \ldots, n$. 
The $d$-componentwise nature of \eqref{eq: ce ps system 2} allows us to write
\begin{align*}
\zeta_j(k) = \zeta_0(k) \left[(D-\lambda I)^{-1} D \textbf{1}\right]_j, \qquad j = 0, \ldots, n,  \quad k = 1, \ldots, d,
\end{align*}
which we abbreviate in the compact notation
\begin{align} \label{eq: ce systems part ii}
\zeta = \zeta_0(D-\lambda I)^{-1} D \textbf{1}.
\end{align}
Substituting \eqref{eq: ce systems part ii} into \eqref{eq: ce ps system 1} gives that \eqref{eq: ce ps system 1}--\eqref{eq: ce ps system 2} has a nontrivial solution if and only if
\begin{align*}
\det \Delta_n(\lambda, \alpha) \neq 0
\end{align*}
with
\begin{align*}
\Delta_n(\lambda, \alpha) = \lambda I_d - L(\alpha) \left(\ell_0 I_d + P(D-\lambda I)^{-1} D \textbf{1} \right).
\end{align*}
If $\det \Delta_n(\lambda, \alpha) =0$, then \eqref{eq: ce ps system 1}--\eqref{eq: ce ps system 2} has a non-trivial solution of the form $(p_\ast, p_\ast (D-\lambda I)^{-1} D \textbf{1})$, where $p_\ast \neq 0$ satisfies $\Delta_n(\lambda, \alpha) p_\ast = 0$. 

Applying Theorem \ref{thm: hopf bif ode} to system \eqref{eq: ps approx system} we obtain: 

\begin{theorem}[Hopf bifurcation in pseudospectral ODE] \label{thm: hopf ps systems}

Consider the system \eqref{eq: dde system} and suppose that Hypothesis \ref{hyp: condities functies system} is satisfied. If there exist $\alpha_n \in \mathbb{R}$ and {\color{black}$\omega_n >0$} such that 
\begin{enumerate}
\item $i \omega_n$ is a simple root of $ \det \Delta_n(\lambda, \alpha_n) =0$;
\item the branch of roots of $\det \Delta_n(\lambda, \alpha) =0$ through $i\omega_n$ at $\alpha = \alpha_n$ intersects the imaginary axis transversally, i.e. the real part of the derivative of the roots along the branch is non-zero. If $p_\ast, q_\ast \in \mathbb{C}^n \backslash \{0 \}$ are vectors such that $\Delta_n(i \omega_n, \alpha_n) p_\ast = 0, \ \Delta_n(i \omega_n, \alpha_n)^T q_\ast = 0$ and $q_\ast \cdot D_1 \Delta_n(i \omega_n, \alpha_n) p_\ast= 1$, then this condition amounts to
\begin{align*}
\re \bigl(q_\ast  \cdot D_2 \Delta_n(i\omega_n,\alpha_n) p_\ast) \neq 0,
\end{align*}
\item {\color{black}$k i \omega_n$ is not a root of $\det  \Delta_n(\lambda, \alpha_n) = 0$ for $k = 0,2,3 \ldots$}
\end{enumerate}
then a Hopf bifurcation occurs for $\alpha = \alpha_n$. 

Moreover, $\alpha^\ast$ as in Theorem \ref{thm: hopf bif ode} has the expansion $\alpha^\ast(\epsilon) =\alpha_n + a_{2n} \epsilon^2 + o(\epsilon^2)$, with $a_{2n}$ given by
\begin{align*}
a_{2n} = \frac{\re c_n}{\re \bigl( q_\ast \cdot D_2 \Delta_n(i\omega_n,\alpha_n) p_\ast \bigr) }
\end{align*}
with
\begin{equation} \label{eq: c ps systems}
\begin{aligned}
c_n = 
&  \frac{1}{2} q_\ast \cdot D_1^3 g(0, \alpha_n) \bigl(P_0 p, P_0 p, P_0\overline{p}\bigr)  \\ & \qquad + q_\ast \cdot D_1^2 g(0, \alpha_n)\Bigl( \Delta_n(0,\alpha_n)^{-1} D_1^2 g(0, \alpha_n) \bigl(P_0 p, P_0\overline{p}\bigr) P_0 \begin{pmatrix}
1 \\ \textbf{1}
\end{pmatrix}, P_0 p \Bigr) \\
&\qquad +  \frac{1}{2} q_\ast \cdot D_1^2 g(0, \alpha_n)  \bigl(
\Delta_n(2 i\omega_n,\alpha_n)^{-1} D_1^2 g(0, \alpha_n)\bigl(P_0 p, P_0 p \bigr) P_0 \begin{pmatrix}
1 \\ (D-2 i\omega_n I)^{-1} D \mathbf{1}
\end{pmatrix}, P_0\overline{p}\Bigr).
\end{aligned}
\end{equation}

with $p = (p_\ast, p_\ast(D-i \omega_n)^{-1} D \textbf{1})$. 

\end{theorem}

Regarding the approximation of the Hopf bifurcation in the pseudospectral scheme, we have the following results (cf Proposition \ref{prop 1} and Proposition \ref{prop 2}):
\begin{prop}
Consider system \eqref{eq: dde system} and assume that the hypotheses of Theorem \ref{thm: hopf bifurcation dde systems} are satisfied. Then for $n \in \mathbb{N}$ large enough, there exist $\alpha_n, \omega_n$ such that $i \omega_n$ is a simple root of $\det \Delta_n(\lambda, \alpha_n) = 0$ and here exists a $C_1 > 0$ such that
\[\left| (\alpha_n, \omega_n) - (\alpha_0, \omega_0) \right| \leq  \frac{1}{\sqrt{n}} \left( \frac{C_1}{n}\right)^n \]
for all $n \in \mathbb{N}$ large enough. 

Assume moreover that for $n$ large enough, $ \det \Delta_n (k i \omega_n, \alpha_n) \neq 0$ for $k = 0, 2, 3 \ldots$. Then the hypotheses of Theorem \ref{thm: hopf ps systems} are satisfied and $\lim_{n \to \infty} a_{2n} = a_{20}$. Moreover, if the nonlinearity $g: X \times \mathbb{R} \to X$ is $C^4$, then there exists a $C_2 > 0$ such that
\begin{align*}
\left| a_{2n} -a_{20} \right| \leq \frac{1}{\sqrt{n}} \left( \frac{C_2}{n}\right)^n
\end{align*}
for all $n \in \mathbb{N}$ large enough. 
\end{prop}

\begin{prop}
Consider system \eqref{eq: dde system} and suppose that {\color{black}there exist a $N_0 \in \mathbb{N}$ such that for $n \in \mathbb{N}, \ n \geq N_0$}, the hypotheses of Theorem \ref{thm: hopf ps systems} are satisfied with $\lim_{n \to \infty} \alpha_n = \alpha_0, \ \lim_{n \to \infty} \omega_n = \omega_0 \neq 0$ and $\lim_{n \to \infty} a_{2n} = a_{20}'$.
Moreover, suppose that 
\begin{enumerate}
\item The sequence $\left(\det D_1 \Delta_n (i \omega_n, \alpha_n)\right)_{n \in \mathbb{N}}$ is uniformly bounded away from zero;
\item If we denote by $p_{\ast}, q_{\ast}$ the vector such that $\Delta_n(i \omega_n, \alpha_n) p_{\ast} = 0, \ \Delta_n(i \omega_n, \alpha_n)^T q_{\ast} = 0$ and \newline $q_{\ast} \cdot D_1 \Delta_n(i \omega_n, \alpha_n) p_{\ast} = 1$, then the sequence $\left(\re \left(q_{\ast} \cdot D_2 \Delta_n (i \omega_n, \alpha_n) p_{\ast} \right) \right)_{n \in\mathbb{N}}$ is uniformly bounded away from zero;
\item For each $k = 0, 2, \ldots$, the sequences $\left(\det \Delta_n (k i \omega_n, \alpha_n) \right)_{n \in \mathbb{N}}$ are uniformly bounded away from zero. 
\end{enumerate}
Then the hypotheses of Theorem \ref{thm: hopf bifurcation dde systems} are satisfied and the direction coefficient is given by $a_{20}'$, i.e. $a_{20} = a_{20}'$. 
\end{prop}

\section{Outlook} \label{sec: outlook}

\begin{figure}[t]
\centering
\includegraphics[scale=0.95]{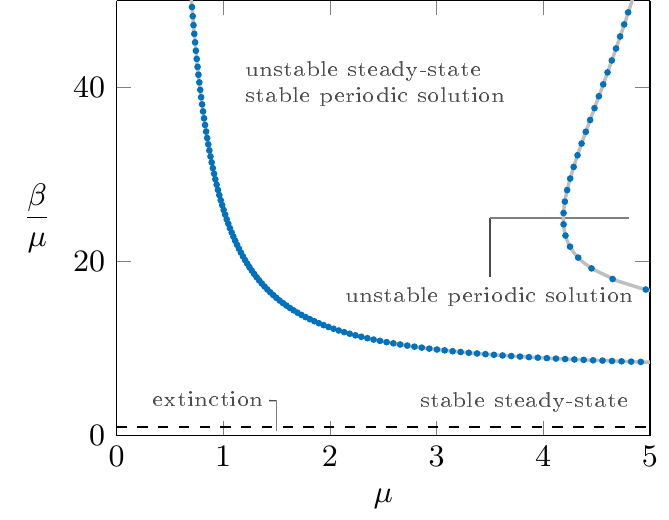}
\includegraphics[scale=0.95]{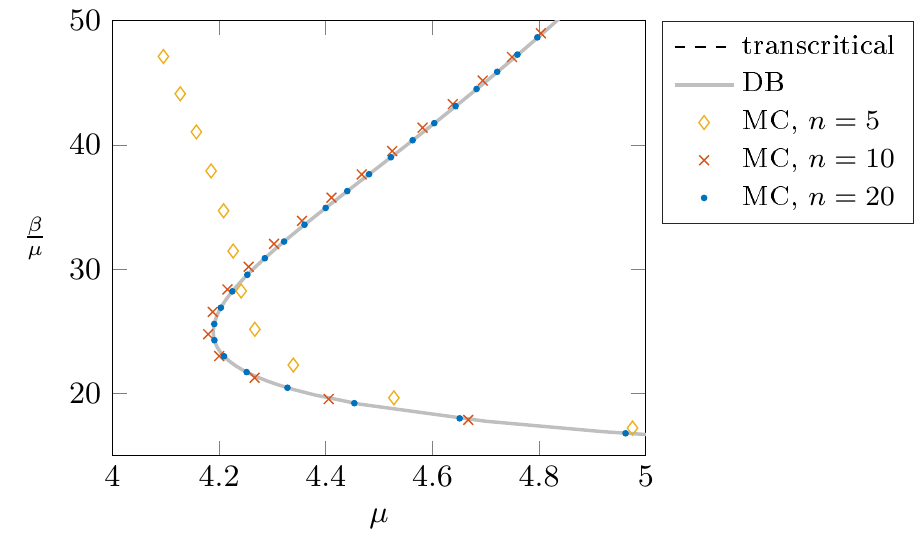}
\caption{
Stability diagram of \eqref{eq: blowfly unscaled} and its pseudospectral approximation for $\tau = 1$ and $h(x) = e^{-x}$. 
The Hopf and period doubling bifurcation curves are approximated numerically with DDE-BIFTOOL (gray solid, DB), and MatCont (colors, MC).
\textcolor{black}{
The right panel focusses on the approximation of the period doubling curve for different dimensions of the ODE system. We can observe the convergence of the approximated curve to that obtained with DDE-BIFTOOL when increasing the dimension $n$ (although larger dimension is required compared to the approximation of the Hopf bifurcation). 
}
}
\label{fig:pd-approx 2}
\label{fig:blowflies-regions}
\end{figure}

\begin{figure}
\centering
\includegraphics[scale=1]{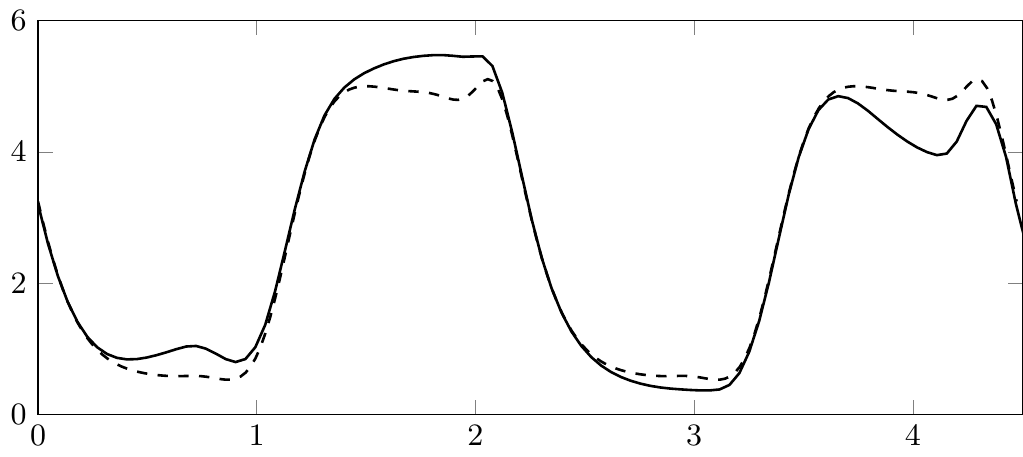}
\caption{Periodic solutions of \eqref{eq: blowfly unscaled}, approximated with MatCont and $n=20$, for $\mu=7$ and $\beta=105$ (after the period doubling bifurcation, which is detected at $\beta \approx 98.22$).
The dashed line shows the periodic solution on the unstable branch (period $T\approx 2.24$); the solid line shows the periodic solution on the stable branch emerging from the period doubling bifurcation (period $T\approx 4.47$).
}
\label{fig:blowflies-per-sol}
\end{figure}

In the Introduction and in Section \ref{sec: blowfly}  we claimed that the combination of pseudospectral discretisation and MatCont enables a reliable bifurcation analysis without requiring excessive computational efforts. Indeed, by using numerical bifurcation software one can push the analysis beyond the Hopf bifurcation and approximate the branch of periodic orbits emerging from Hopf, as well as its bifurcations.
The DDE \eqref{eq: blowfly unscaled}, which has only one discrete point delay, can be directly analysed also by existing and well-established numerical software for delay {\color{black}differential} equations, like DDE-BIFTOOL. 
We indeed use DDE-BIFTOOL as a benchmark for validating the output of the pseudospectral discretisation. 
In Figure \ref{fig:blowflies-regions} we show more detailed stability regions of equation \eqref{eq: blowfly unscaled} in the plane $(\mu,\frac{\beta}{\mu})$, including not only the Hopf bifurcation curve, but also the curve of period doubling bifurcations, approximated with DDE-BIFTOOL (version 3.1) and MatCont (version 7p1), running on Matlab 2019a.
At the period doubling bifurcation, the branch of periodic solutions originating from the Hopf point switches stability and becomes unstable, whereas a new stable branch of periodic solutions arises. The stability change is observed from the approximated multipliers at the periodic orbit, with one multiplier exiting the unit circle and crossing -1 as $\beta$ increases.
Two examples of coexisting periodic solutions are plotted in Figure \ref{fig:blowflies-per-sol}, taken from the unstable and stable branches.

In both the package DDE-BIFTOOL and MatCont, each periodic orbit is approximated via collocation of a boundary value problem in the period interval (see for example \cite{po, Alessia}). This requires the specification of a number of discretisation intervals and the degree of the collocation polynomial in each interval (we stress however that such mesh and polynomial degree are different from and independent of the mesh points and polynomial degree used to discretise the delay interval in the pseudospectral approach). In all the computations of this section we have taken a piecewise mesh of 40 intervals in the period interval, and polynomial approximations of degree 4 in each interval. These values guarantee sufficient accuracy in the approximation of the periodic orbits, so that the dominating errors in Figure \ref{fig:pd-approx 2} are those due to the chosen polynomial degree of the pseudospectral approximation.

As a further illustration we consider the system of equations
\begin{align}
w'(t) &= 1 - \frac{kw(t)w(t-1)}{2}q(t) \label{eq:system 1} \\
q'(t) &= w(t)-c, \label{eq:system 2}
\end{align}
for $k,c \in \mathbb{R}_+$. {\color{black}Equations \eqref{eq:system 1} -- \eqref{eq:system 2} correspond to a fluid flow of information between sender and receiver; $w$ refers to the average size of the sent information packages, $q$ to the average queue length and the total roundtrip time has been normalised to $1$ \cite{hollot, niculescu}. }

The stability regions in the plane $(k,c)$ are plotted in Figure \ref{fig:system-regions}: the lower curve represents the Hopf bifurcation, whereas the upper curve is a period doubling bifurcation.
Two periodic solutions are plotted in Figure \ref{fig:system-per-sol}.

\medskip

Numerical software like MatCont, among their output parameters, normally return also the value of the first Lyapunov coefficient at the Hopf bifurcation. 
We remark, however, that the output of MatCont applied to the pseudospectral approximation can not be directly taken as approximation of the direction coefficient of the DDE, since the scaling of the left and right eigenvectors traditionally used for ODE differs from the scaling used for DDE. For DDE, indeed, the eigenvectors are scaled by taking the first component equal to 1, whereas for ODE systems the eigenvector is normalised by requiring the 2-norm to be equal 1.

\medskip 

So far we did not manage to treat the non-resonance condition in a completely satisfactory manner, and we explicitly assumed condition \eqref{eq: non-resonance}. For retarded functional differential equations, there are no roots of the characteristic equation high up the imaginary axis. So checking the non-resonance condition is executable. One would expect that for the approximating pseudospectral ODE systems similar bounds can be found, but our initial (and somewhat half-hearted) attempt to derive them failed. When the dimension of the ODE system increases, so does the number of roots. Numerical observations (also in other contexts) suggest that these `additional' roots have real parts moving towards minus infinity. In particular, they do not even come close to the imaginary axis. For the `trivial’ DDE $y'(t) = 0$, where the `spurious’ eigenvalues are simply the eigenvalues of the matrix $D$, it is indeed proved that they go to minus infinity when the dimension increases \cite{eigenvaluesps, eigenvaluesps2}. For more general DDE, one could try to prove that the number of roots to the right of any vertical line in the complex plane is preserved if the dimension of the approximation is large enough (in the spirit of the preservation of the dimension of the unstable manifold treated for instance in \cite{rignum}). As far as we know, there are as yet no theoretical results for the pseudospectral approximation considered here. 

\medskip

The (numerical) bifurcation theory of delay equations is well developed, see for instance \cite{maikelsebastiaan} and the references given there. Our analysis of the Hopf bifurcation can be seen as a proof of principle that pseudospectral approximation yields a reliable bifurcation diagram, a reliable `picture'. But checking the details case by case for the entire catalogue of bifurcations would, we think, provide only negligible additional  insight. An attractive alternative might be to try to show, as a next step, that the centre manifold of a delay equation is (in a sense to be specified) approximated by the centre manifold of the pseudospectral ODE system. 

\medskip

The technical difficulties of state-dependent delay equations disappear in the pseudospectral approximation, for the very simple reason that polynomials are infinitely many times differentiable. So while here we focused on showing that known results for delay equations are well approximated by corresponding results for pseudospectral ODE, we might try to \emph{prove} results for state-dependent delay equations by showing that the limit of results for pseudospectral ODE systems exists and provides information about (behaviour of) solutions of the delay equation. A concrete challenge would be to provide a rigorous underpinning for the results derived in \cite{sieber}.

\begin{figure}
\centering
\includegraphics[scale=1]{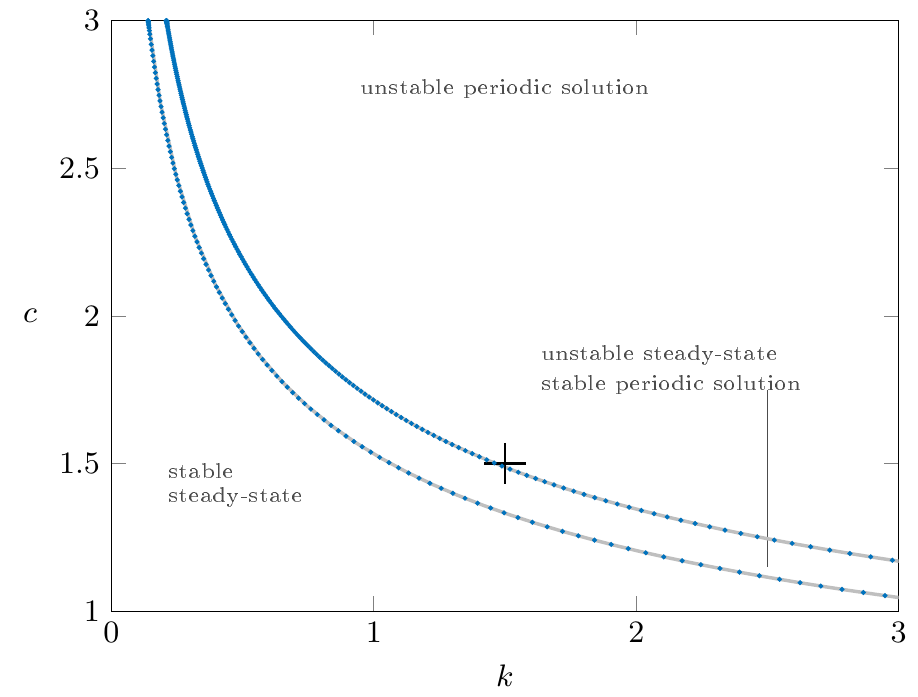}
\caption{Stability regions of system \eqref{eq:system 1}--\eqref{eq:system 2} {\color{black}and its pseudospectral approximation}, approximated with DDE-BIFTOOL (gray curve) and MatCont with $n=20$ (blue dots).
The lower curve corresponds to the Hopf bifurcation, the upper curve to the period doubling bifurcation.
}
\label{fig:system-regions}
\end{figure}

\begin{figure}
\centering
\includegraphics[scale=1]{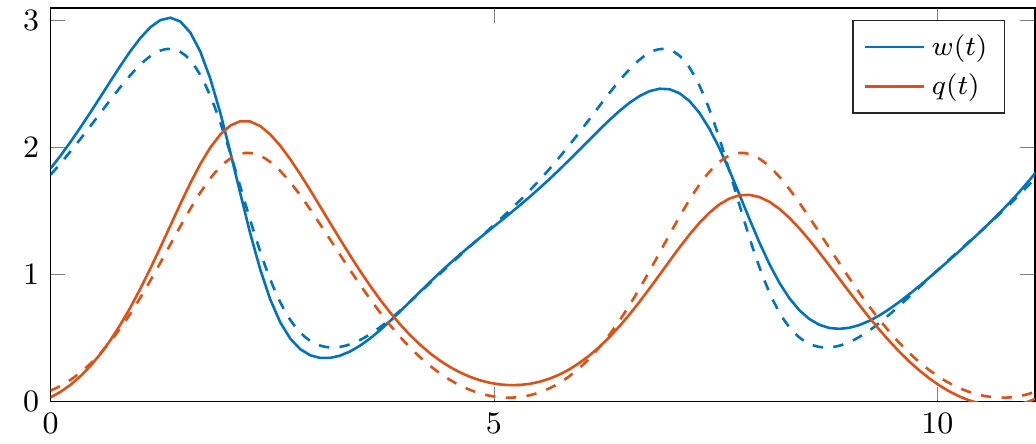}
\caption{Periodic solutions of system \eqref{eq:system 1}--\eqref{eq:system 2}, approximated with MatCont and $n=20$, for $c=k=1.5$ (beyond the period doubling bifurcation).
The dashed line shows the periodic solution on the unstable branch (period $T\approx 5.57$); the solid line shows the periodic solution on the stable branch emerging from the period doubling bifurcation (period $T\approx 11.15$).
}
\label{fig:system-per-sol}
\end{figure}

\appendix

\section{Stability charts for the `Nicholson's blowflies' equation} \label{appendix}

We collect some results concerning the DDE
\begin{equation} \label{eq: blowfly scaled}
N'(t) = - \mu N(t) + \beta N(t-1) h(N(t-1))
\end{equation}
with parameters $\beta, \mu \geq 0$. We pay special attention to the case
\begin{equation} \label{eq: h}
h(x) = e^{-x}. 
\end{equation}
Equation \eqref{eq: blowfly unscaled} can be brought in the form \eqref{eq: blowfly scaled} by scaling of time with a factor $\tau$. This entails the introduction of dimensionless parameters 
\begin{equation*}
\mu_{new} = \tau \mu_{old}, \qquad \beta_{new} = \tau \beta_{old},
\end{equation*}
where ``new'' refers to \eqref{eq: blowfly scaled} and ``old'' refers to \eqref{eq: blowfly unscaled}. Note, incidentally, that $\beta_{old}$ also incorporates the survival of the juvenile period and that one can make this explicit by putting
\begin{equation*}
\beta_{old} = \beta_0 e^{-\tau \mu_{old}};
\end{equation*}
but we will not elaborate on this further. Finally, note that the case $h(x) = e^{- \sigma x}$ can be reduced to \eqref{eq: h} by scaling of $N$ with a factor $\sigma$. 

In \cite{note} it is argued that using two parameters in Hopf bifurcation studies has great advantages. As \eqref{eq: blowfly scaled} naturally has two parameters, we are in the ideal situation. 

Nontrivial steady states $\overline{N}$ of \eqref{eq: blowfly scaled} are characterised by the equation 
\begin{equation} \label{eq: steady state}
h \left(\overline{N}\right) = \frac{\mu}{\beta}.
\end{equation}
Under the assumptions
\begin{itemize}
\item $h(0) = 1$;
\item $h$ is monotonically decreasing;
\item $\lim_{x \to \infty} h(x) = 0$
\end{itemize}
equation \eqref{eq: steady state} has a unique positive solution for $\beta > \mu$. In the parameter plane the line $\beta = \mu$ corresponds to a transcritical bifurcation. For $\beta< \mu$ the population goes extinct. For $\beta$ slightly larger than $\mu$, the nontrivial steady state is asymptotically stable. Our first aim is to investigate whether or not $\overline{N}$ can lose its stability by way of a Hopf bifurcation. See also \cite{blowflyhopf2} for an analysis of the occurrence of a Hopf bifurcation in system \eqref{eq: blowfly scaled} and \cite{blowflyhopf1} for an analysis of the direction of this bifurcation. 

As a first step we put 
\begin{equation*}
N(t) = \overline{N}+ x(t)
\end{equation*}
and rewrite \eqref{eq: blowfly scaled} as
\begin{equation*}
x'(t) = b_1 x(t) + b_2 x(t-1)  + \mathcal{G}(x(t-1), \mu, \beta)
\end{equation*}
where
\begin{equation} \label{eq: b}
b_1 = - \mu, \qquad b_2 = \beta (h(\overline{N}) + \overline{N} h'(\overline{N}))
\end{equation}
and
\begin{equation} \label{eq: g}
\mathcal{G}(x, \mu, \beta) = \beta \overline{N} \left(h(\overline{N}+x)- h( \overline{N}) - h'(\overline{N}) x\right) + \beta \left(h (\overline{N} + x) - h(\overline{N}) \right) x. 
\end{equation}
So the characteristic equation corresponding to the linearised equation reads
\begin{equation} \label{eq: ce blowfly}
\lambda - b_1 - b_2 e^{-\lambda} = 0. 
\end{equation}
This equation is analysed in great detail in \cite[Section XI.2]{sunstarbook}, to which we refer for justification of some statements below. 

Substituting $\lambda = i \omega$ into \eqref{eq: ce blowfly} and solving for $b_1$ and $b_2$ we obtain
\begin{equation} \label{eq: stability boundary}
b_1 = \frac{\omega \cos \omega}{\sin \omega}, \qquad b_2 = -\frac{\omega}{\sin \omega}. 
\end{equation}
The stability region in the $(b_1, b_2)$-plane is bounded by the line 
\begin{equation*}
b_1 + b_2 = 0, \qquad b_1 \leq 1
\end{equation*}
(corresponding to $\lambda = 0$ being a root of \eqref{eq: ce blowfly}) and the curve defined by \eqref{eq:  stability boundary} with
\begin{equation} \label{eq: omega}
0 \leq \omega < \pi. 
\end{equation}
Note that the curve and the line intersect at $(b_1 ,b_2) = (1, -1)$ corresponding to $\lambda = 0$ being a double root of \eqref{eq: ce blowfly}.  The root $\lambda = i \omega$ is simple for $\omega >0$. 

If one follows a one-parameter path in the $(b_1, b_2)$-plane that crosses the curve defined by \eqref{eq: stability boundary}, \eqref{eq: omega} transversally, the root of \eqref{eq: ce blowfly} crosses the imaginary axis transversally. 

There are no roots on the imaginary axis if $(b_1, b_2)$ is not of the form \eqref{eq: stability boundary}. By adjusting the domain of definition of $\omega$, one obtains via \eqref{eq: stability boundary} countably many curves in the $(b_1, b_2)$-plane such that \eqref{eq: ce blowfly} has a root on the imaginary axis. These curves do not intersect the curve corresponding to \eqref{eq: omega} nor each other. We conclude that the non-resonance condition is satisfied. We refer to \cite[Figure XI.1, page 306]{sunstarbook} for a graphical summary. 

The next step is to translate the results from the $(b_1, b_2)$-plane to the $(\mu, \beta)$-plane or, for that matter, the $(\mu, \beta/\mu)$-plane. Here it becomes useful to adopt \eqref{eq: h} since in that case \eqref{eq: b} amounts to
\begin{equation*}
b_1 = - \mu, \qquad b_2 = \mu \left(1- \ln \left(\frac{\beta}{\mu}\right) \right)
\end{equation*}
with inverse
\begin{equation} \label{eq: transformation}
\mu = - b_1, \qquad \beta = - b_1 e^{1 + \frac{b_2}{b_1}}. 
\end{equation}
By combining  \eqref{eq: stability boundary}, \eqref{eq: omega} and \eqref{eq: transformation} we obtain the curve depicted in Figure \ref{fig:hopf-analytic}, albeit in the $(\mu, \beta/\mu)$-plane. Note, however, that the interpretation requires $\mu \geq 0$ and that accordingly we should restrict to $\pi/2 \leq \omega \leq \pi$. 

The conclusion is that if we follow a one-parameter path in the $(\mu, \beta)$- or $(\mu,\beta/\mu)$-plane that crosses the stability boundary transversally, all assumptions of Theorem \ref{thm: hopf bifurcation dde} are satisfied. 

\medskip

We now compute the stability boundaries for the pseudospectral approximation to \eqref{eq: blowfly scaled}. The pseudospectral approximation to \eqref{eq: blowfly scaled} reads
\begin{equation} 
\label{eq: blowfly ps}
\begin{aligned}
y_0'(t) &= - \mu y_0(t) + \beta y_n(t)h(y_n(t)) \\
y'(t) &= D y(t) - D \textbf{1} y_0(t)
\end{aligned}
\end{equation}
where we have written $(y_0, \ldots, y_n) = (y_0, y) \in \mathbb{R}^{n+1}$. Equilibria of \eqref{eq: blowfly scaled} are in one-to-one correspondence with equilibria of \eqref{eq: blowfly ps}, so \eqref{eq: blowfly ps} has a non-trivial equilibrium $\overline{N} \textbf{1}$ with $h(\overline{N}) = \mu/\beta$ for $\beta > \mu$.  We shift the non-trivial equilibrium to zero via the coordinate transform $(y_0, y) = \overline{N} \textbf{1} + (x_0, x)$; then \eqref{eq: blowfly ps} becomes
\begin{equation} \label{eq: blowfly ps shift}
\begin{aligned}
x_0'(t) &=b_1 x_0(t) + b_2 x_n(t) + \mathcal{G}(x_n(t), \mu, \beta) \\
x'(t)& = D x(t) - D \textbf{1} x_0(t)
\end{aligned}
\end{equation}
with $b_1, b_2$ and $\mathcal{G}$ defined in \eqref{eq: b}--\eqref{eq: g}. The characteristic equation corresponding to the linearisation of \eqref{eq: blowfly ps shift} becomes (cf \eqref{eq: ce ode})
\begin{equation} \label{eq: ce blowfly ps}
\lambda - b_1 - b_2\left[(D-\lambda I)^{-1} D \textbf{1}\right]_n = 0. 
\end{equation}
We compute the stability boundary by setting $\lambda = i \omega$ and solving for $b_1, b_2$:
\begin{equation} \label{eq: stability boundary ps}
b_1 = - \frac{\omega \re \left[(D-i \omega I)^{-1} D \textbf{1}\right]_n}{\im \left[(D-i \omega I)^{-1} D \textbf{1}\right]_n}, \qquad b_2 = \frac{\omega}{\im \left[(D-i \omega I)^{-1} D \textbf{1}\right]_n}.
\end{equation}
Note that the expressions for $b_1, b_2$ have singularities but at different values than the expressions for $b_1, b_2$ in \eqref{eq: stability boundary}. By defining $h_0(x) = -\sin(x)$ and $h_n(x) = \im \left[(D-i x)^{-1} D \textbf{1}\right]_n$ {\color{black}and applying Lemma \ref{lem: discrete IFT}}, we see that the singularities of $b_1, b_2$ defined in \eqref{eq: stability boundary ps} approximate the singularities of $b_1, b_2$ defined in \eqref{eq: stability boundary}. Moreover, the expressions \eqref{eq: stability boundary ps} converge to the expressions \eqref{eq: stability boundary} for $n \to \infty$ and for $\omega$ in compact intervals; see Figure \ref{fig:blowfly_n=345}. 

We now want to determine whether the root $i \omega$ crosses the imaginary axis transversely if we cross the curves \eqref{eq: stability boundary ps} transversely. For ease of computation we restrict to varying $b_2$. If $i \omega$ is a simple root of \eqref{eq: ce blowfly ps}, then it lies on a branch of roots $\lambda(b_2)$ and the derivative along this branch is given by
\begin{equation} \label{eq: crossing ps}
{\color{black} \lambda'(b_2) =  \frac{\left[(D- i \omega I)^{-1} D \textbf{1}\right]_n}{1 - b_2\left[(D- i \omega I)^{-2} D \textbf{1}\right]_n}.}
\end{equation}
with $b_2$ defined in \eqref{eq: stability boundary ps}. So if the real part of the right hand side of \eqref{eq: crossing ps} is non-zero, the root on the imaginary axis crosses transversely if we vary $b_2$. 

Note that by Lemma \ref{lem: convergence collocation} and Corollary \ref{cor: convergence derivatives ce}, $i \omega$ is a simple zero of \eqref{eq: ce blowfly ps} for $n$ large enough; moreover, the expression in \eqref{eq: crossing ps} is non-zero for $n$ large enough. However, for fixed values of $n$ one has to check these conditions explicitly. We now do this for the case $n = 2$. 

\medskip For $n = 2$, the matrices $D$ and $A_2$ are given by
\begin{equation} \label{eq: explicit}
D = \begin{pmatrix}
0 & -1 \\
4 & - 3
\end{pmatrix}, \qquad A_2 = \begin{pmatrix}
b_1 & 0 & b_2 \\
1 & 0 & -1 \\
-1 & 4 & -3
\end{pmatrix}.
\end{equation}
We first compute the characteristic equation for the eigenvalues of $A_2$. With $D$ as in \eqref{eq: explicit}, \eqref{eq: ce blowfly ps} becomes
\begin{equation} \label{eq: roots blowfly 1}
 \lambda - b_1 - b_2 \frac{4-\lambda}{\lambda^2 + 3 \lambda +4} = 0.
\end{equation}
As a sanity check, we compute the eigenvalues of $A_2$ as roots of $\det (\lambda I - A_2) = 0$. We find that the eigenvalues are roots of the equation
\begin{equation} \label{eq: roots blowfly 2}
 (\lambda - b_1)\left(\lambda^2 + 3 \lambda + 4\right)-b_2(4-\lambda) = 0. 
\end{equation}
and indeed we see that the roots of \eqref{eq: roots blowfly 1} are exactly the roots of \eqref{eq: roots blowfly 2}. 

We now compute the stability boundary. Equation \eqref{eq: roots blowfly 2} has a root $\lambda = 0$ if
\begin{align} \label{eq: zero root}
b_1 = -b_2
\end{align}
(the fact that steady states of DDE and the approximating ODE are in one-to-one correspondence guarantees that steady state bifurcation conditions are too). 
Substituting $\lambda = i \omega$ in \eqref{eq: roots blowfly 1} and solving for $b_1, b_2$ (or, equivalently, computing \eqref{eq: stability boundary ps} for $D$ as in \eqref{eq: explicit}) gives
\begin{equation} \label{eq: stability boundary explicit}
b_1(\omega) = \frac{7 \omega^2-16}{\omega^2-16}, \qquad b_2(\omega) = \omega^2 - 4 + 3 \cdot \frac{7 \omega^2-16}{\omega^2-16}. 
\end{equation}
Note that the expressions for $b_1, b_2$ have singularities at $\omega = \pm 4$; the stability region in the $(b_1, b_2)$-plane is bounded by the line \eqref{eq: zero root} and the curve define by \eqref{eq: stability boundary explicit} with 
\begin{equation} \label{eq: omega n=2}
-4 \leq \omega \leq 4;
\end{equation}
see Figure \ref{fig:hopf_approx_n=2}.  

If we cross the curve \eqref{eq: stability boundary explicit}, \eqref{eq: omega n=2} by varying $b_2$, we find that the derivative of the eigenvalue along the branch is given by
{\color{black}
\begin{equation} \label{eq: lambda prime}
\begin{aligned}
\lambda'(\omega) &= \frac{i \omega - 4}{-3 \omega^2 + 6 i \omega + 4 - 2 i \omega b_1(\omega) - 3 b_1(\omega)+b_2(\omega)} \\
& = \frac{i \omega-4}{\omega\left(-2 \omega + 6i - 2b_1(\omega) i\right)} 
\end{aligned}
\end{equation}}
The real part of the denominator of \eqref{eq: lambda prime} is non-zero for $\omega \neq 0$; hence the denominator of  \eqref{eq: lambda prime} is non-zero for $\omega \neq 0$, which means that $\omega \neq 0$ is a simple zero of \eqref{eq: roots blowfly 2} for $b_1, b_2$ defined in \eqref{eq: stability boundary explicit}. The real part of \eqref{eq: lambda prime} becomes
\begin{equation*}
\re \lambda'(\omega) = \frac{14-2b_1(\omega)}{4 \omega^2 + (6-2b_1(\omega))^2}.
\end{equation*}
On the interval $(-4, 4)$ the expression for $b_1$ in \eqref{eq: stability boundary explicit} attains its maximum $b_1 = 1$ for $\omega = 0$. Therefore $\re \lambda'(\omega) \neq 0$ along the curve \eqref{eq: stability boundary explicit}--\eqref{eq: omega n=2}. Moreover, since $A_2$ has exactly three eigenvalues (counting multiplicity), the non-resonance condition is in this case easy to check. A resonance between eigenvalues $i \omega$ and $k i \omega$, $k> 0$, would require 4 eigenvalues and can therefore not happen. A resonance between $i \omega$, $\omega > 0$ and 0 can also not happen because the curve defined by \eqref{eq: stability boundary explicit} with $\omega \neq 0$ does not intersect the curve $b_1 = -b_2$. So the conclusion is that if we cross the stability boundary \eqref{eq: stability boundary explicit} transversally, a Hopf bifurcation of system \eqref{eq: blowfly ps} with $n = 2$ occurs. 

\medskip

For higher values of $n$, we can also explicitly compute the stability boundary $(b_1, b_2)$ as defined in \eqref{eq: stability boundary ps}. For $n = 3$, the characteristic equation becomes
\begin{equation*}
\lambda - b_1 -b_2 \frac{3 \lambda^2 - 32 \lambda+96}{3 \lambda^3 + 19 \lambda^2 + 64 \lambda + 96} = 0
\end{equation*}
and the stability boundary as defined in \eqref{eq: stability boundary ps} becomes
\begin{equation} \label{eq:n=3}
b_1(\omega) = 17 + \frac{2048 \bigl(7 \omega^2-72\bigr)}{9 \omega^4-1088 \omega^2+9216}, \qquad 
b_2(\omega) = -\frac{9 \omega^6-23 \omega^4+448 \omega^2+9216}{9 \omega^4-1088 \omega^2+9216}.
\end{equation}
For $n \ge 4$, the formula's can still be computed explicitly in terms of the mesh points $\theta_j$ but become rather long. Furthermore, for $n  \ge 4$, we need numerical approximations for $\theta_j$ to plot the parametric curves.

We have plotted the stability boundary \eqref{eq:n=3} together with \eqref{eq: zero root} in Figure  \ref{fig:hopf_approx_n=3}. Note that the curves defined by  \eqref{eq:n=3} and \eqref{eq: zero root} do not self intersect and do not intersect each other; so there is never a resonance between two roots on the imaginary axis. Moreover, we see that Figure \ref{fig:hopf_approx_n=3} has an extra curve compared to Figure \ref{fig:hopf_approx_n=2}. So it seems that the infinite number of curves defined by \eqref{eq: stability boundary} get approximated one by one as we increase the discretisation index $n$. 

In Figure \ref{fig:blowfly_n=345} we have plotted the graphs of the functions defined by \eqref{eq: stability boundary ps} for $n = 3,4,5$. We see that for $n = 3,4$, there are two curves within the depicted window. We see that as $n$ increases, the curves within the depicted window lie closer together. For $n =5$ a third curve appears in the window. 

\medskip

For the case where $h$ is given as in \eqref{eq: h}, we analyse the Lyapunov coefficient along the stability boundary for the DDE \eqref{eq: blowfly scaled}. For $\pi/2 < \omega < \pi$, define the functions
\begin{equation*}
B_{10}(\omega) = \frac{e^{-i \omega}}{1+ b_2(\omega)e^{-i \omega}}, \qquad B_{20}(\omega) = \frac{e^{-2 i \omega}}{2 i \omega - b_1(\omega) -b_2(\omega) e^{-2i \omega}} B_{10}(\omega)
\end{equation*}
with $b_1(\omega), b_2(\omega)$ as defined in \eqref{eq: stability boundary}. Then $c_0$ as defined in \eqref{eq: c dde} becomes
\begin{equation} \label{eq: c blowfly}
c_0 = \frac{1}{2} D_1^3 \mathcal{G}(0, \mu, \beta) B_{10}(\omega) - \frac{(D_1^2 \mathcal{G}(0, \mu, \beta))^2}{b_1 + b_2} B_{10}(\omega) + \frac{1}{2}(D_1^2 \mathcal{G}(0, \mu, \beta))^2 B_{20}(\omega)
\end{equation}
with 
\begin{equation} \label{eq: derivatives}
D_1^2 \mathcal{G}(0, \mu, \beta) = \mu \ln \left(\frac{\beta}{\mu}\right)-2 \mu, \qquad D_1^3 \mathcal{G}(0,\mu, \beta) = - \mu \ln \left(\frac{\beta}{\mu}\right)+3 \mu. 
\end{equation}
For $\pi/2 < \omega < \pi$, $\re c_0$ is plotted in Figure \ref{fig:coeff}. Note in particular that $\re c_0$ is always negative along the stability boundary \eqref{eq: stability boundary}--\eqref{eq: omega}. 

To compute the Lyapunov coefficient of the system \eqref{eq: blowfly ps shift} when $h$ is given by \eqref{eq: h}, define the functions
\begin{equation*}
B_{1n}(\omega) = \frac{\left( ((D-i \omega I)^{-1} D \textbf{1})_n \right)^2 \left((D+i \omega I)^{-1}D \textbf{1} \right)_n}{1-b_2(\omega) \left((D-i\omega I)^{-2} D \textbf{1}\right)_n}, \quad B_{2n}(\omega)   = \frac{\left((D-2 i \omega I)^{-1} D \textbf{1}\right)_n}{2 i \omega - b_1 - b_2 \left((D-i \omega I)^{-1} D \textbf{1}\right)_n} B_{1n}(\omega)
\end{equation*}
with $b_1(\omega), b_2(\omega)$ defined in \eqref{eq: stability boundary ps}. Then $c_n$ defined in \eqref{eq: c ps} becomes
\begin{equation} \label{eq: c blowfly ps}
c_n = \frac{1}{2} D_1^3 \mathcal{G}(0,\mu, \beta) B_{1n}(\omega) - \frac{(D_1^2 \mathcal{G}(0,\mu, \beta))^2}{b_1 + b_2} B_{1n}(\omega) + \frac{1}{2}(D_1^2 \mathcal{G}(0,\mu, \beta))^2 B_{2n}(\omega)
\end{equation}
with $D_1^2 \mathcal{G}(0,\mu, \beta), \ D_1^3 \mathcal{G}(0,\mu, \beta)$ defined in \eqref{eq: derivatives}. For $n = 1,2$, we have plotted $\re c_n$ in Figure \ref{fig:coeff}. We note that both for $n =2$ and $n = 3$ the Lyapunov coefficient is negative. This reinforces our earlier conclusions that already for low values of $n$, we find good qualitative agreement between the behaviour of the DDE and the pseudospectral ODE.

\begin{figure}[t]
\centering
\includegraphics[scale=0.5]{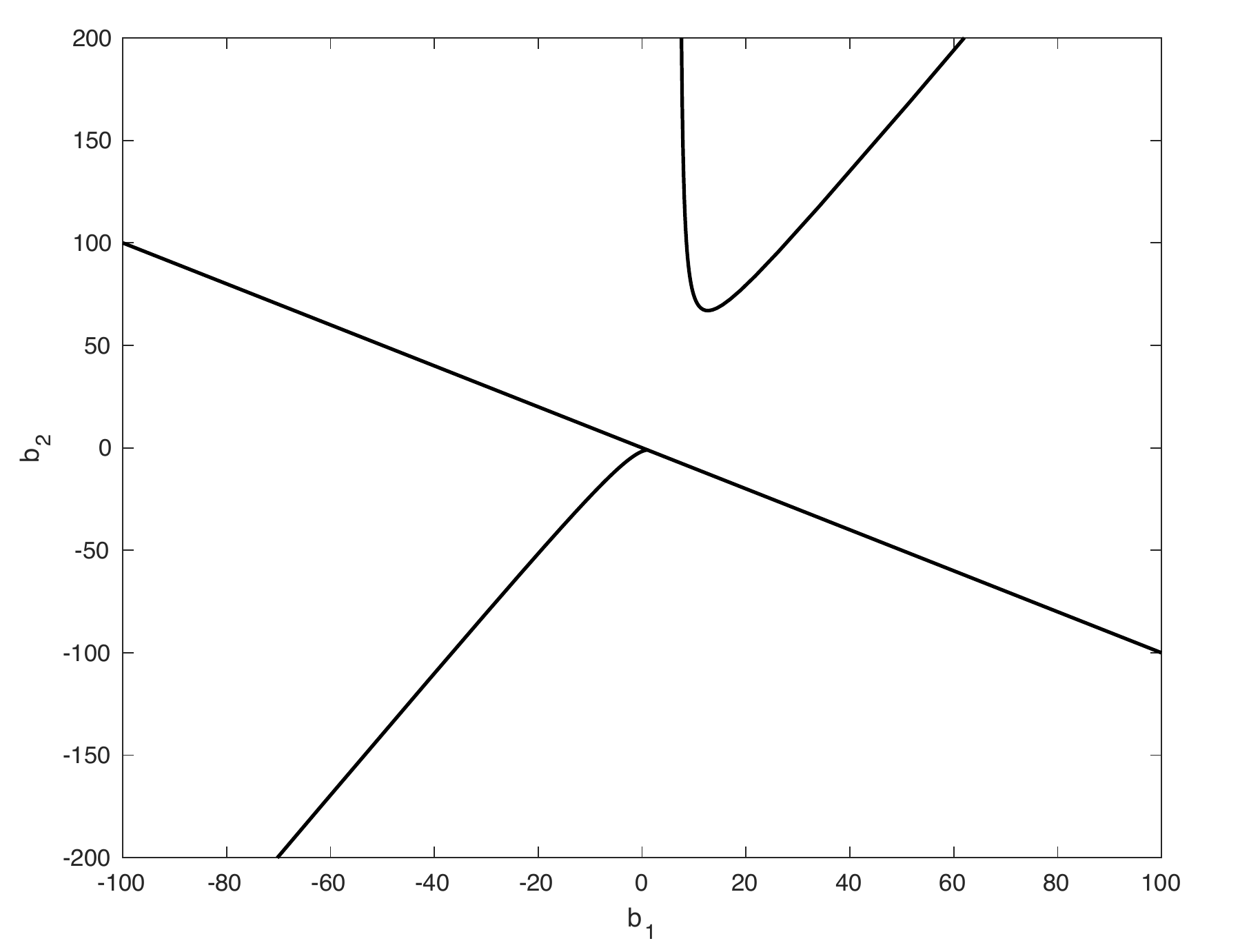}
\caption{The curves defined by \eqref{eq: zero root}, \eqref{eq: stability boundary explicit}.}
\label{fig:hopf_approx_n=2}
\end{figure}

\begin{figure}[t]
\centering
\includegraphics[scale=0.5]{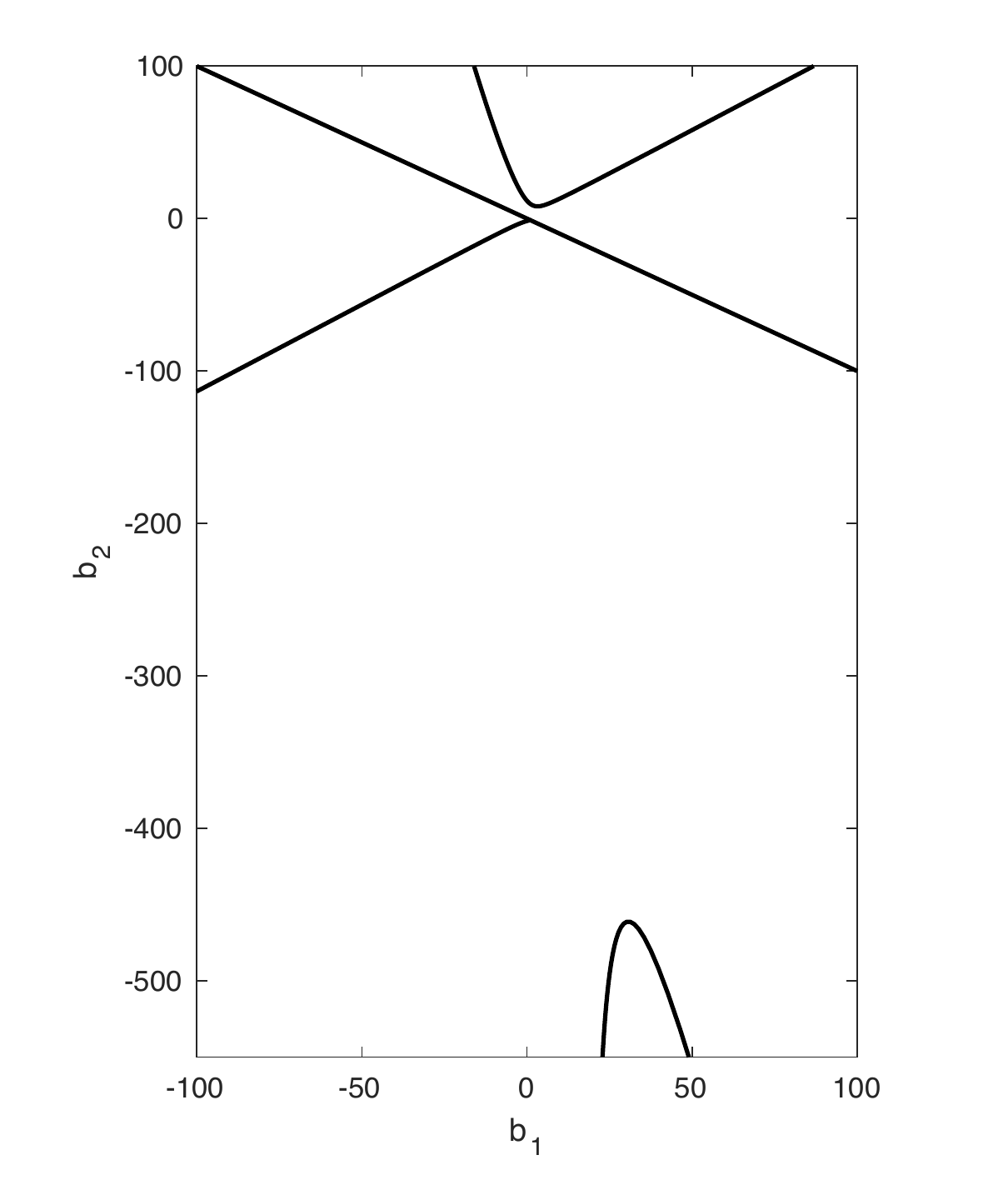}
\caption{The curves defined by \eqref{eq: zero root}, \eqref{eq:n=3}.}
\label{fig:hopf_approx_n=3}
\end{figure}

\begin{figure}[h] 
\begin{center}
\includegraphics[width=0.4\textwidth, angle=0]{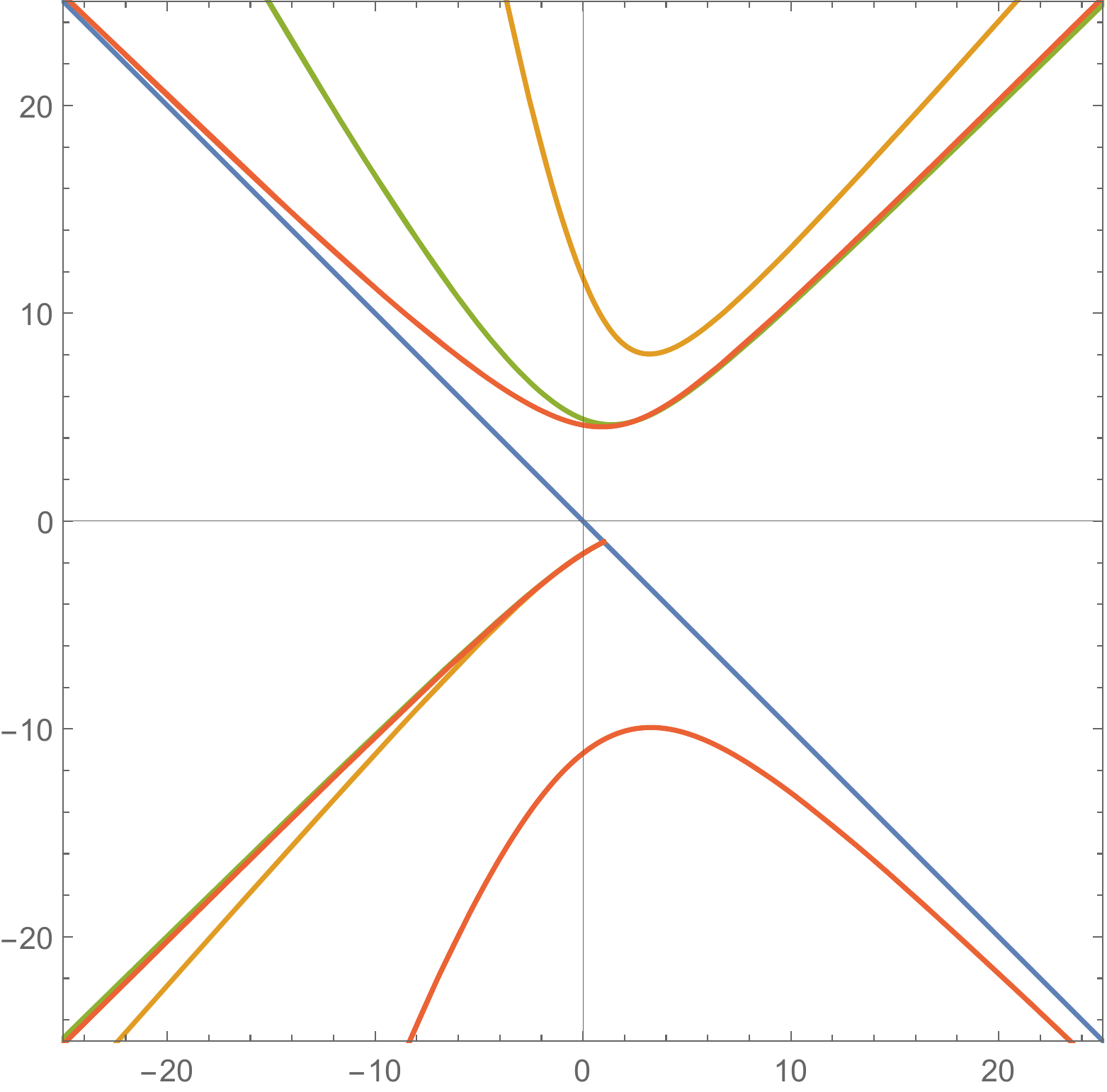}
\end{center}
\caption{Parametric plot of the graphs of the functions defined by \eqref{eq: stability boundary ps} for different values of $n$ 
in the $(b_1,b_2)$-plane: $n=3$ (brown, light, see expression \eqref{eq:n=3}), $n=4$ (green) and $n=5$ (brown, dark). The blue line corresponds to the line defined by \eqref{eq: zero root}.}
\label{fig:blowfly_n=345}
\end{figure}

\begin{figure}[t]
\centering
\includegraphics[scale=0.5]{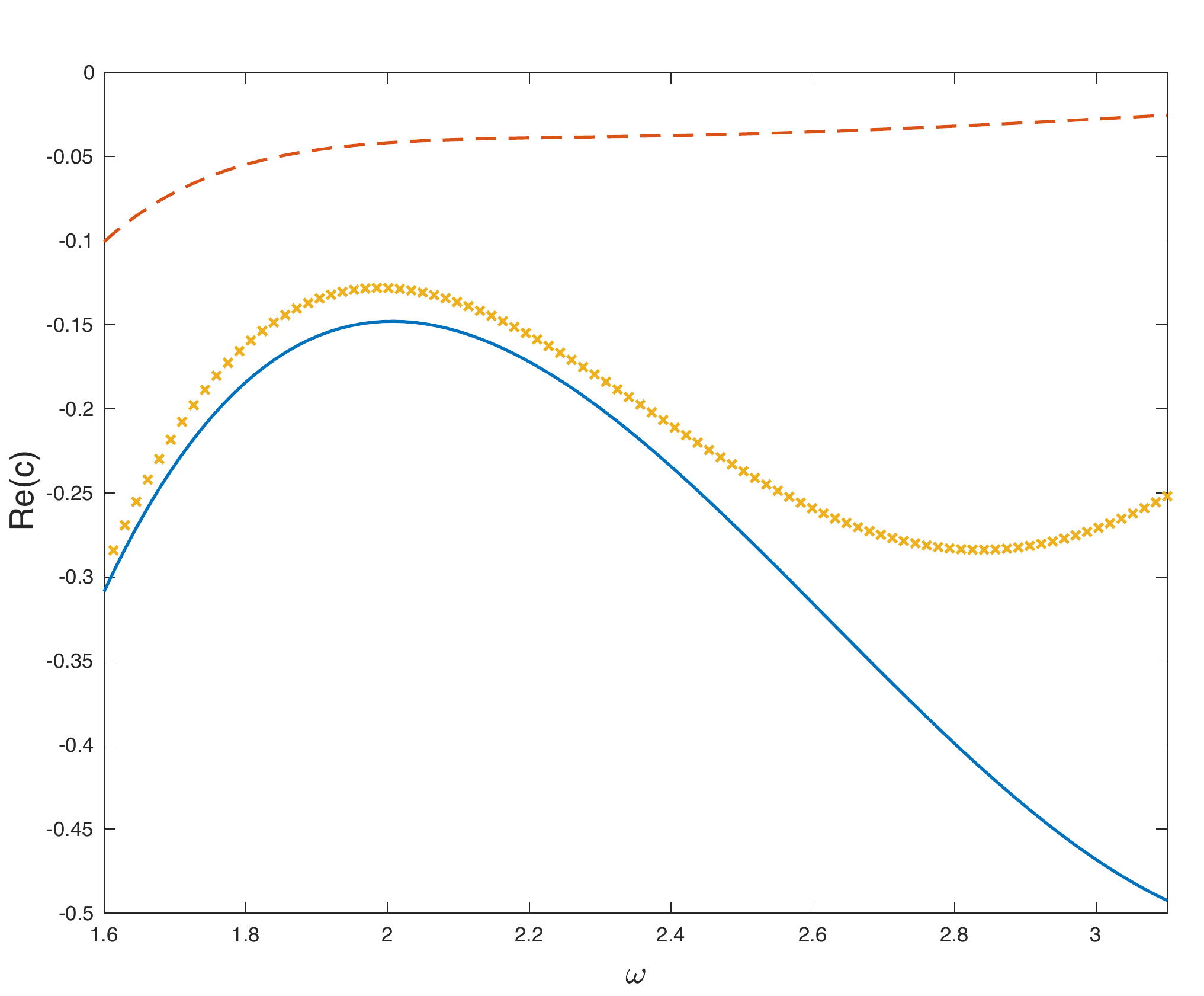}
\caption{The Lyapunov coefficient \eqref{eq: c blowfly} (blue line), and the Lyapunov coefficient \eqref{eq: c blowfly ps} for $n = 2$ (orange dashed line) and $n = 3$ (yellow crosses).}
\label{fig:coeff}
\end{figure}

\bibliographystyle{plain}
\bibliography{references}

\begin{thebibliography}{10}

\bibitem{auto}
Auto.

\bibitem{knut}
Knut.

\bibitem{Alessia}
A.~And{\`o} and D.~Breda.
\newblock Convergence analysis of collocation methods for computing periodic
  solutions of retarded functional differential equations.
\newblock {\em SIAM J. Numer. Anal.}, to appear, 2020.

\bibitem{berrut}
J-P. Berrut and L.~Trefethen.
\newblock Barycentric lagrange interpolation.
\newblock {\em SIAM Review}, 46:501--517, 2004.

\bibitem{maikelsebastiaan}
M.M. Bosschaert, S.G. Janssens, and Yu.A. Kuznetsov.
\newblock Switching to nonhyperbolic cycles from codimension two bifurcations
  of equilibria of delay differential equations.
\newblock {\em SIAM J. Appl. Dyn. Syst.}, 19:252--303, 2020.

\bibitem{nonlinearps}
D.~Breda, O.~Diekmann, M.~Gyllenberg, F.~Scarabel, and R.~Vermiglio.
\newblock Pseudospectral discretization of nonlinear delay equations: New
  prospects for numerical bifurcation analysis.
\newblock {\em SIAM Journal on Applied Dynamical Systems}, 15:1--23, 2016.

\bibitem{renewalbif}
D.~Breda, O.~Diekmann, D.~Liessi, and F.~Scarabel.
\newblock Numerical bifurcation analysis of a class of nonlinear renewal
  equations.
\newblock {\em Electron. J. Qual. Theory of Differ. Equ.}, 65:1--24, 2016.

\bibitem{bmv}
D.~Breda, S.~Maset, and R.~Vermiglio.
\newblock Pseudospectral differencing methods for characteristic roots of delay
  differential equations.
\newblock {\em SIAM J. Sci. Comput.}, 27:482--495, 2005.

\bibitem{bmvboek}
D.~Breda, S.~Maset, and R.~Vermiglio.
\newblock {\em Stability of linear delay differential equations: a numerical
  approach with MatLab}.
\newblock Springer, 2014.

\bibitem{matcont}
A.~Dhooge, W.~Govaerts, Yu.A. Kuznetsov, H.G.E. Meijer, and B.~Sautois.
\newblock New features of the software matcont for bifurcation analysis of
  dynamical systems.
\newblock {\em Math. Comput. Model. Dyn. Syst.}, 14:147--175, 2008.

\bibitem{infinitedelay}
O.~Diekmann, M.~Gyllenberg, and J.~Metz.
\newblock Finite dimensional state representation of linear and nonlinear delay
  systems.
\newblock {\em J. Dynam. Differential Equations}, 30:439–1467, 2018.

\bibitem{note}
O.~Diekmann and K.~Korvasov{\'a}.
\newblock A didactical note on the advantage of using two parameters in hopf
  bifurcation studies.
\newblock {\em J. Biol. Dyn.}, 7:21--30, 2013.

\bibitem{ps-sunstar}
O.~Diekmann, F.~Scarabel, and S.~Vermiglio.
\newblock Pseudospectral discretisation of delay differential equations in
  sun-star formulation: results and conjectures.
\newblock {\em Discrete Contin. Dyn. Syst. Ser. S}, 13:2575–2602, 2020.

\bibitem{sunstarbook}
O.~Diekmann, S.~van Gils, S.~Verduyn Lunel, and H.-O. Walther.
\newblock {\em Delay equations: Functional-, Complex-, and Nonlinear Analysis}.
\newblock Springer, 1995.

\bibitem{eigenvaluesps}
M.~Dubiner.
\newblock Asymptotic analysis of spectral methods.
\newblock {\em J. Sci. Comput.}, 2:3--31, 1987.

\bibitem{po}
K.~Engelborghs, T.~Luzyanina, K.~J. in~’t Hout, and D.~Roose.
\newblock Collocation methods for the computation of periodic solutions of
  delay differential equations.
\newblock {\em SIAM J. Sci. Comput.}, 22:1593--1609, 2001.

\bibitem{biftool1}
K.~Engelborghs, T.~Luzyanina, and D.~Roose.
\newblock Numerical bifurcation analysis of delay differential equations using
  dde-biftool.
\newblock {\em ACM Trans. Math. Softw.}, 28:1--21, 2002.

\bibitem{biftool2}
K.~Engelborghs, T.~Luzyanina, and G.~Samaey.
\newblock Dde-biftool v. 2.00: a matlab package for bifurcation analysis of
  delay differential equations.
\newblock Technical Report TW-330, Department of Computer Science, K.U.Leuven,
  2001.

\bibitem{statedep}
Ph. Getto, M.~Gyllenberg, Y.~Nakata, and F.~Scarabel.
\newblock Stability analysis of a state-dependent delay differential equation
  for cell maturation: analytical and numerical methods.
\newblock {\em J. Math. Bio.}, 79:281--328, 2019.

\bibitem{bifboek}
W.~Govaerts.
\newblock {\em Numerical Methods for Bifurcations of Dynamical Equilibria}.
\newblock SIAM, 2000.

\bibitem{blowflies}
W.~Gurney, S.~Blythe, and R.~Nisbet.
\newblock Nicholson's blowflies revisited.
\newblock {\em Nature}, 287:17--21, 1980.

\bibitem{infinite}
M.~Gyllenberg, F.~Scarabel, and R.~Vermiglio.
\newblock Equations with infinite delay: numerical bifurcation analysis via
  pseudospectral discretization.
\newblock {\em Appl. Math. Comput.}, 333:490--505, 2018.

\bibitem{hollot}
C.~V. Hollot and Y.~Chait.
\newblock Nonlinear stability analysis for a class of tcp/aqm networks.
\newblock {\em Proceedings of the 40th IEEE Conference on Decision and
  Control}, 3:2309–2314, 2001.

\bibitem{kuznetsov}
Yu.~A. Kuznetsov.
\newblock {\em Elements of Applied Bifurcation Theory}.
\newblock Springer, 4th edition, 2004.

\bibitem{Lani-Wayda13}
B.~Lani-Wayda.
\newblock Hopf bifurcation for retarded functional differential equations and
  for semiflows in banach spaces.
\newblock {\em J. Dynam. Differential Equations}, 4:1159–1199, 2013.

\bibitem{rignum}
J-P Lessard and M.~James.
\newblock A functional analytic approach to validated numerics for eigenvalues
  of delay equations.
\newblock {\em Journal of Computational Dynamics}, 7:123, 2020.

\bibitem{MastroianniBook}
G.~Mastroianni and G.~Milovanovi\'c.
\newblock {\em Interpolation {P}rocesses. {B}asic {T}heory and {A}pplications}.
\newblock Springer, 2008.

\bibitem{niculescu}
S.-I. Niculescu and K.~Gu.
\newblock {\em Advances in time-delay systems}.
\newblock Springer, 2012.

\bibitem{ift}
C.~Poetzsche.
\newblock Numerical dynamics of integrodifference equations: global
  attractivity in a $c^0$-setting.
\newblock {\em SIAM J. Numer. Anal.}, 5:2121--2141, 2019.

\bibitem{interpolation}
T.~Rivlin.
\newblock {\em An Introduction to the Approximation of Functions}.
\newblock Blaisdell, 1969.

\bibitem{pde}
F.~Scarabel, D.~Breda, O.~Diekmann, M.~Gyllenberg, and R.~Vermiglio.
\newblock Numerical bifuration analysis of physiologically structured
  population models via pseudospectral approximation.
\newblock {\em Vietnam Journal of Mathematics}, 2020.

\bibitem{blowflyhopf2}
H.~Shu, L.~Wang, and J.~Wu.
\newblock Global dynamics of nicholson’s blowflies equation revisited: Onset
  and termination of nonlinear oscillations.
\newblock {\em J. Differential Equations}, 255:2565–2586, 2013.

\bibitem{sieber}
J.~Sieber.
\newblock Local bifurcations in differential equations with state-dependent
  delay.
\newblock {\em Chaos}, 27:114326, 2017.

\bibitem{trefethen}
L.~Trefethen.
\newblock {\em Spectral methods in MATLAB}.
\newblock SIAM, 2000.

\bibitem{eigenvaluesps2}
J.~Wang and F.~Waleffe.
\newblock The asymptotic eigenvalues of first-order spectral differentiation
  matrices.
\newblock {\em J. Appl. Math. Phys.}, 2:176--188, 2014.

\bibitem{blowflyhopf1}
J.~Wei and M.~Li.
\newblock Hopf bifurcation analysis in a delayed nicholson blowflies equation.
\newblock {\em Nonlinear Anal.}, 60:1351 – 1367, 2005.

\end{thebibliography}

\end{document}